\documentclass[10pt, letterpaper]{article}
\usepackage{amsmath}
\usepackage{amssymb}
\usepackage{amsfonts}
\usepackage{amscd}
\usepackage{mathrsfs}
\usepackage{amsthm}

\newcommand{\CC}{{\rm\bf C}}
\newcommand{\RR}{{\rm\bf R}}
\newcommand{\QQ}{{\rm\bf Q}}
\newcommand{\ZZ}{{\rm\bf Z}}

\newcommand{\Adeles}{{\rm\bf A}}

\DeclareMathOperator{\Spec}{\mathrm{Spec}}

\DeclareMathOperator{\GL}{\mathrm{GL}}

\DeclareMathOperator{\Oo}{\mathrm {O}}
\DeclareMathOperator{\U}{\mathrm {U}}
\DeclareMathOperator{\SO}{\mathrm {SO}}

\DeclareMathOperator{\Aut}{\mathrm {Aut}}

\DeclareMathOperator{\Hom}{\mathrm {Hom}}

\DeclareMathOperator{\kernel}{\mathrm {ker}}
\DeclareMathOperator{\image}{\mathrm {im}}

\DeclareMathOperator{\Gal}{\mathrm {Gal}}
\DeclareMathOperator{\ind}{\mathrm{Ind}}

\DeclareMathOperator{\res}{\mathrm{Res}}

\DeclareMathOperator{\Lie}{\mathrm{Lie}}

\DeclareMathOperator{\sgn}{\mathrm{sgn}}

\newcommand{\liea}{{\mathfrak {a}}}

\newcommand{\liec}{{\mathfrak {c}}}

\newcommand{\lieg}{{\mathfrak {g}}}
\newcommand{\lieh}{{\mathfrak {h}}}

\newcommand{\liek}{{\mathfrak {k}}}
\newcommand{\liel}{{\mathfrak {l}}}

\newcommand{\lieq}{{\mathfrak {q}}}

\newcommand{\lieu}{{\mathfrak {u}}}

\newcommand{\liegk}{{\mathfrak {gk}}}

\theoremstyle{Theorem}
\newtheorem{introconjecture}{Conjecture}

\newtheorem{introtheorem}[introconjecture]{Theorem}

\theoremstyle{plain}
\newtheorem{theorem}{Theorem}[section]
\newtheorem{lemma}[theorem]{Lemma}
\newtheorem{corollary}[theorem]{Corollary}
\newtheorem{proposition}[theorem]{Proposition}
\newtheorem{conjecture}[theorem]{Conjecture}

\theoremstyle{remark}
\newtheorem{remark}[theorem]{Remark}

\newcommand{\mscfootnote}[1]{{%
  \let\thempfn\relax
  \footnotetext[0]{\emph{#1}}
}}

\begin{document}

\title{On Period Relations for Automorphic~$L$-functions~I}
\author{Fabian Januszewski}
\date{}

\maketitle

\mscfootnote{2010 Mathematics Subject Classification. Primary: 11F67; Secondary: 11F41, 11F70, 11F75, 22E55.}

\begin{abstract}
  This paper is the first in a series of two dedicated to the study of period relations of the type
  $$
    L(\frac{1}{2}+k,\Pi)\;\in\;(2\pi i)^{d\cdot k}\Omega_{(-1)^k}\QQ(\Pi),\quad \frac{1}{2}+k\;\text{critical},
  $$
    for certain automorphic representations $\Pi$ of a reductive group $G.$ In this paper we discuss the case $G=\GL(n+1)\times\GL(n).$ The case $G=\GL(2n)$ is discussed in part two. Our method is representation-theoretic and relies on the author's recent results on global rational structures on automorphic representations. We show that the above period relations are intimately related to the field of definition of the global representation $\Pi$ under consideration. The new period relations we prove are in accordance with Deligne's Conjecture on special values of $L$-functions and we expect our method to apply to other cases as well.
\end{abstract}

{
\tableofcontents
}

\section*{Introduction}

In his solution to the Basel Problem, Euler showed in \cite{euler1937} that the Riemann $\zeta$-function has the remarkable property that
$$
\zeta(2k)\in (2\pi i)^{2k}\QQ^\times,
$$
for all integers $k\geq 1$. Thanks to the work of Riemann \cite{riemann1859}, we also know that this is valid even at the (non-critical) value $2k=0$.

Euler's result found many generalizations, to Dedekind $\zeta$-functions, Dirichlet- and Hecke $L$-functions and $L$-functions attached to modular cusp forms \cite{manin1972,shimura1976}, and more generally Hilbert modular forms \cite{manin1976,shimura1977,shimura1978}. In his book \cite{shimurabook2000}, Shimura extended these results to Siegel modular forms and forms on unitary groups. Recently, Shimura's method has been taken up in the context of orthogonal groups by Furusawa and Morimoto \cite{furusawamorimoto2014,furusawamorimoto2016}. In all these cases it is possible to evaluate the corresponding archimedean zeta integrals explicitly.

Outside the context of Shimura varieties, no such relation is known for Euler products of higher degrees attached to general linear groups, although we expect this to be true by a celebrated Conjecture of Deligne \cite{deligne1979} on special values of $L$-functions.

In this paper, we develop a general method which allows to approach this problem via the cohomology of arithmetic groups combined with representation theoretic methods and exemplify the technique in the case of Rankin-Selberg $L$-functions for $\GL(n+1)\times\GL(n)$ over number fields. We thereby generalize Euler's relation to Euler products of Rankin-Selberg type of higher degree. In the sequel \cite{januszewskipart2}, we apply our method to $\GL(2n)$ where we obtain another generalization of Euler's period relation.

Fix a number field $F$ and consider two irreducible cuspidal regular algebraic representations $\Pi_1$ and $\Pi_2$ of $\GL(n+1)$ and $\GL(n)$ over $F$ respectively. Assume that the cohomological weights of the pair $(\Pi_1,\Pi_2)$ are balanced in the sense of section \ref{sec:arithmeticity}. For reasons that become clear below, it is appropriate to consider $\Pi_1$ and $\Pi_2$ as a single representation $\Pi:=\Pi_1\widehat{\otimes}\Pi_2$ of the reductive group
$$
G\;=\;\res_{F/\QQ}\GL(n+1)\times \res_{F/\QQ}\GL(n),
$$
over $\QQ.$ We let $L(s,\Pi)=L(s,\Pi_1\times\Pi_2)$ denote the finite part of the Rankin-Selberg $L$-function attached to $\Pi$ in the sense of Jacquet, Shalika and Piatetski-Shapiro \cite{jpss1983}, i.e.\ $L(s,\Pi)$ is an Euler product over all finite places without the $\Gamma$ factors.

In this context, the automorphic analogue of Deligne's Conjecture predicts that for each $k\in\ZZ$ such that $s=\frac{1}{2}+k$ is critical for $L(s,\Pi),$
\begin{equation}
L(\frac{1}{2}+k,\Pi)\;\in\;(2\pi i)^{m\cdot k}\cdot\Omega_{(-1)^k}(\Pi)\cdot\QQ(\Pi)
\label{eq:introperiodrelation}
\end{equation}
with two complex constants $\Omega_\pm(\Pi)\in\CC^\times$ independent of $k$,
$$
m\;=\;\frac{(n+1)n}{2}[F:\QQ],
$$
and $\QQ(\Pi)=\QQ(\Pi_1,\Pi_2)$ denotes Clozel's field of rationality of $\Pi$ (known to be a number field \cite{clozel1990}).

We have a well known cohomological definition for periods $\Omega_\pm(\Pi,s_0)\in\CC^\times,$ such that \eqref{eq:introperiodrelation} holds where a priori the periods $\Omega_{(-1)^k}(\Pi)=\Omega_{(-1)^k}(\frac{1}{2}+k,\Pi)$ {\em vary} with $k$.

The behavior of the periods $\Omega_\pm(\Pi,s_0)\in\CC^\times$ under finite order twists has been thoroughly studied by many authors and is now well understood \cite{hida1994,schmidt1993,kazhdanmazurschmidt2000,kastendiss2007,kastenschmidt2008,raghuramshahidi2008,raghuram2010,raghuramtanabe2011,januszewski2011,januszewski2014,januszewski2015,raghuram2015}.

However, at least since Kasten's dissertation \cite{kastendiss2007,kastenschmidt2008} in 2007, where for the first time rationality results in the presence of several critical values were established for $L(s_0,\Pi_1\times(\Pi_2\otimes\chi))$ for $n\geq 2,$ it wasn't clear how to approach the dependence of these cohomological periods in the critical variable $s_0$.

Our first main result is (cf.\ Theorem \ref{thm:globalrankinperiods} in the text)

\begin{introtheorem}\label{thm:introglobalrankinperiods}
  Let $F/\QQ$ denote a number field and let $n\geq 1.$ If $n\geq 3$ or $n\geq 2$ if $F$ admits a complex place, assume the validity of Conjecture \ref{conj:rankinfunctional0} or of Conjecture \ref{conj:rankinfunctional} for a balanced weight $\mu$ of $G.$

  Let $\Pi=\Pi_1\widehat{\otimes}\Pi_2$ be an irreducible cuspidal regular algebraic automorphic representation of $\GL_{n+1}(\Adeles_F)\times\GL_n(\Adeles_F)$ of balanced weight $\mu.$ Then there exist non-zero periods $\Omega_\pm$, numbered by the $2^{r_F}$ characters $\pm$ of $\pi_0(F_\infty^\times)$, such that for each critical half integer $s_0=\frac{1}{2}+k$ for the Rankin-Selberg $L$-function $L(s,\Pi_1\times\Pi_2),$ and each finite order Hecke character
$$
\chi:\quad F^\times\backslash\Adeles_F^\times\to\CC^\times,
$$
we have, in accordance with Deligne's Conjecture (cf.\ Conjecture \ref{conj:deligne}),
$$
\frac{L(s_0,\Pi_1\times(\Pi_2\otimes\chi))}
{G(\overline{\chi})^{\frac{(n+1)n}{2}}(2\pi i)^{k[F:\QQ]\frac{(n+1)n}{2}}
\Omega_{(-1)^{k}\sgn\chi}}
\;\in\;
\QQ_K(\Pi_1,\Pi_2,\chi).
$$
Furthermore, the expression on the left hand side is $\Aut(\CC/\QQ_K)$-equivariant. Here $\QQ_K\subseteq\RR$ is a number field depending only on $F$ and possibly $n,$ and $\QQ_K=\QQ$ for $F$ totally real or a CM field.
\end{introtheorem}

The field $\QQ_K/\QQ$ is the field of definition of a rational model of a maximal compact subgroup in $G(\RR)$ with certain properties (i.e.\ a $\QQ_K(\sqrt{-1})/\QQ_K$-{\em admissible} model in the sense of \cite{januszewski2017}, cf.\ section \ref{sec:compactmodels}).

Conjectures \ref{conj:rankinfunctional0} and \ref{conj:rankinfunctional} provide two different statements which are both sufficient to prove \eqref{eq:introperiodrelation}.

We show that for $n=1,$ both Conjectures \ref{conj:rankinfunctional0} and \ref{conj:rankinfunctional} are always true for all $F$ and all weights and thereby give an entirely representation theoretic proof for Hida's first Main Theorem in \cite{hida1994} for $\GL(2)$ over arbitrary number fields.
We refer the reader to \cite{raghuramtanabe2011} for a representation-theoretic treatment of the Manin-Shimura period relation for Hilbert modular cusp forms from \cite{manin1976,shimura1978}.

In Theorem \ref{thm:conjectures}, we establish Conjectures \ref{conj:rankinfunctional0} and \ref{conj:rankinfunctional} for $n=2$ and $F$ totally real in all balanced weights and thus prove that \eqref{eq:introperiodrelation} holds for the degree $6$ Rankin-Selberg $L$-function for $\GL(3)\times\GL(2)$ over totally real fields $F,$ in which case our period relation is new. Combined with the Gelbart-Jacquet lift, Theorem \ref{thm:introglobalrankinperiods} therefore implies the analogous rationality result for triple products of Hilbert modular forms ${\bf f}\otimes{\bf f}\otimes{\bf g}$ in the balanced case (we refer to \cite{furusawamorimoto2014,furusawamorimoto2016} for the imbalanced case of a general triple product).

We provide further evidence for Conjectures \ref{conj:rankinfunctional0} and \ref{conj:rankinfunctional} in the sequel \cite{januszewskipart2}, where we reduce both conjectures to the continuity of a certain cohomologically induced functional. In loc.\ cit.\ we also establish their analogues in the case of $\GL(2n),$ which allows us to prove the corresponding period relation in this case.

Our approach relies on our construction of global rational structures in \cite{januszewski2017}. $\Pi$ carries a natural global rational structure defined over $\QQ(\Pi).$ Upon restriction to $G(\RR)^0\times G(\Adeles^{(\infty)}),$ $\Pi$ decomposes naturally into a sum
$$
\Pi\;=\;\bigoplus_\pm \Pi_\pm,
$$
where $\pm$ runs again over the $2^{r_F}$ characters of the component group $\pi_0(F_\infty^\times).$ Now the field of definition of $\Pi_\pm$ may be {\em larger} than that of $\Pi$ itself. Indeed, using our results from \cite{januszewski2017}, we prove our second main result.

\begin{introtheorem}\label{thm:introglobalrational}
If $\sqrt{-1}\in\QQ_K(\Pi)$, then $\Pi_\pm$ is defined over $\QQ_K(\Pi)$. Otherwise, $\Pi_\pm$ is defined over $\QQ_K(\Pi)$ if and only if
\begin{equation}
4\mid [F:\QQ](n+1)n.
\label{eq:introdivisibility}
\end{equation}
In all other cases it is defined over $\QQ_K(\Pi,\sqrt{-1}).$
\end{introtheorem}

On the one hand, the divisibility relation \eqref{eq:introdivisibility} in Theorem \ref{thm:introglobalrational} stems ultimately from the fact that a root system of type $B_n$ or $D_n$ admits the negated long Weyl element $-w_0$ as a non-trivial automorphism if and only if it is of type $D_n$ with $n$ odd. In all cases, including root systems of type $A_n$ corresponding to complex places, we exploit that the Brauer obstructions for the groups $\SO(n)$ and $\U(n)$ always vanish.

On the other hand, the divisibility relation \eqref{eq:introdivisibility} is equivalent to
\begin{equation}
i^{[F:\QQ]\frac{(n+1)n}{2}}\;\in\;\QQ,
\label{eq:introiexponent}
\end{equation}
and the left hand side of \eqref{eq:introiexponent} is the contribution of \lq{}$i$\rq{} to \eqref{eq:introperiodrelation} as conjectured by Deligne.

This remarkable coincidence between the symmetry of simple root systems and Deligne's Conjecture is at the heart of our proof of Theorem \ref{thm:introglobalrankinperiods}.

Recently, Harder and Raghuram obtained partial results towards the period relation \eqref{eq:introperiodrelation} in \cite{harder2010,harderraghuram2011,harderraghuram2014preprint} for more general Rankin-Selberg convolutions by exploiting the structure of Eisenstein cohomology: Harder and Raghuram show that the ratio of $L$-values $L(s_0,\Pi_1\times\Pi_2)/L(s_0+1,\Pi_1\times\Pi_2)$ is constant up to rational factors and given by a cohomological period independent of $s_0$. We believe that a proper modification of our method using the results on rational structures in section \ref{sec:cohomologicalinduction} may allow for the computation of the desired period up to rational factors.

{\bf Acknowledgement.} This work, as well as its sequel \cite{januszewskipart2}, is based on the author's Habilitationsschrift \cite{januszewskischrift}. It is the author's great pleasure to thank Claus-G\"unther Schmidt and Stefan K\"uhnlein from Karlsruhe, Don Blasius and Haruzo Hida from UCLA, and Anton Deitmar from T\"ubingen, for being available as members of the author's Habilitation committee. The author thanks A.~Raghuram for very valuable comments and Michael Harris for point him to independent results of Jie Lin from her Dissertation, in which she establishes special cases of Theorem \ref{thm:introglobalrankinperiods} for values in the range of absolute convergence under different hypotheses (cf.\ \cite{lin2016}).

\section{Notation and setup}

We let $\overline{\QQ}\subseteq\CC$ denote the algebraic closure of $\QQ$ inside $\CC$. If $v$ is a place of a field $E$, we write $E_v$ for its completion at $v$. If $E/\QQ$ is a number field, we write $\Adeles_E=\Adeles\otimes_\QQ E$ for the topological ring of adeles over $E$, where $\Adeles$ denotes the ring of adeles over $\QQ$. We write $\Adeles_E^{(\infty)}=\Adeles^{(\infty)}\otimes_\QQ E$ for the ring of finite adeles.

If $E/F$ is a finite separable field extension, we write $\res_{E/F}$ for the functor of restriction of scalars \`a la Weil for the extension $E/F$, sending quasi-projective varieties over $E$ to quasi-projective varieties over $F$. Since this functor preserves finite products, it sends group objects to group objects.

If $G$ is an algebraic or a topological group, we denote its connected component of the identity by $G^0$. We assume henceforth without further mention that the identity component of any linear algebraic group $G$ is always {\em geometrically connected}, i.e.\ the geometrically connected component is already defined over the base field under consideration. Then its component group $\pi_0(G)=G/G^0$ is well defined independently of the base field under consideration. If $\lieg$ is the Lie algebra of $G$, we let $U(\lieg)$ denote its universal enveloping algebra. If $G$ is an algebraic group defined over a field $E$, then $\lieg$ and $U(\lieg)$ are defined over $E$ as well. We adopt the same notation in the context of Lie groups. In this case $U(\lieg)$ is defined over $\RR$, but we usually implicitly consider its complexification in this case.

We write
$$
X(G)\;=\;\Hom(G,\GL_1)
$$
for the group of rational characters of $G$. If $G$ is defined over a field $E$, we denote by $X_E(G)$ the subgroup of characters which are defined over $E$.

A superscript $(\cdot)^\vee$ on a (rational / admissible) representation of $G$, of $G(\RR)$ or a $(\lieg,K)$-module denotes its (rational / admissible) dual.

We assume without loss of generality that all Haar measures on totally disconnected groups we consider have the property that the volumes of compact open subgroups are {\em rational numbers}.

In the body of the paper, $F/\QQ$ denotes a number field admitting $r_F^\RR$ real and $r_F^\CC$ complex places.

\subsection{Algebraic groups}

We are interested in products $G$ of copies of 
$$
G_{n}\;=\;\res_{F/\QQ}\GL_n.
$$
We have a decomposition into a quasi-product
\begin{equation}
G_1\;=\;\res_{F/\QQ}(\GL_1)\;=\;G^{\rm s}_1\cdot G^{\rm an}_1,
\label{eq:rationaltorusdecomposition}
\end{equation}
where $G^{\rm s}_1$ is the maximal $\QQ$-split torus and $G^{\rm an}_1$ is the maximal $\QQ$-anisotropic subtorus in $G_1$. The latter is a quasi-complement of $G^{\rm s}_1$ in $G_1$, i.e. $
G^{\rm s}_1\cap G^{\rm an}_1$ is finite. The projection
$$
p_{\rm s}:\quad G_1\to G_1/G^{\rm an}_1\cong\GL_1
$$
corresponds to the Norm character $N_{F/\QQ}: F^\times\to\QQ^\times,$ and its composition with the determinant induces a character
$$
\mathcal N:=p_{\rm s}\circ\res_{F/\QQ}\det:\quad G_n\to\GL_1.
$$
To describe its behavior on the real points, we introduce for each real archimedean place $v$ of $F$ the local {\em sign character}
$$
\sgn_v:\quad \GL_n(F_v)\;\to\;\RR^\times,\quad h_v\;\mapsto\; \frac{\det(h_v)}{|\det(h_v)|_v}.
$$
By abuse of notation we also consider $\sgn_v$ and the norm $|\cdot|_v$ (implicitly composed with the determinant) as a character of $G_n(\RR)$. Then on real points we have the sign character
$$
\sgn_\infty:=(\otimes_{v\;\text{real}}\sgn_v)\otimes(\otimes_{v\;\text{cplx}}{\bf1}_v):\quad G_n(\RR)\;\to\;\RR^\times,
$$
where ${\bf1}_v:G_n(F_v)\to\RR^\times$ denotes the trivial character. To simplify notation, we drop these trivial factors in the sequel from the notation. Likewise, we have the archimedean norm character $|\cdot|_\infty:=\otimes_v |\cdot|_v: G_1(\RR)\to\RR^\times.$ Then for all $h\in G_n(\RR),$
\begin{equation}
\mathcal N(h)\;=\;
\sgn_\infty(h)\cdot|h|_\infty.
\label{eq:algebraicarchimedeannorm}
\end{equation}
We remark that the group $\Hom(G_n(\RR),\CC^\times)$ of quasi-characters of $G_n(\RR)$ is, as a complex manifold, a disjoint union of $2^{r_F^\RR}\;=\;\#\pi_0(G_n(\RR))$ copies of $\CC$. Each component corresponds uniquely to a finite order character of $G_n(\RR)$ and each such character is of the form
$$
\sgn_\infty^\delta:=\otimes_{v\;\text{real}} \sgn_v^{\delta_v}:\quad G_n(\RR)\to\CC^\times,
$$
with $\delta=(\delta_v)_{v\;\text{real}}\;\in\;\prod_v\{0,1\}.$ Then the component of $\sgn_\infty^\delta$ is parametrized by the charts
\begin{equation}
\CC\ni s\;\mapsto\;\omega^\delta_s\;:=\;\sgn_\infty^\delta\otimes|\cdot|_\infty^s\in\Hom(G_n(\RR),\CC^\times).
\label{eq:quasicharacterchart}
\end{equation}
We call $\omega_s^\delta$ {\em algebraic} whenever $s\in\ZZ$ (this notion does not agree with Andr\'e Weil's notion of being of \lq{}type $A_0$\rq{}). If $\chi$ is a quasi-character of $G_n(\RR)$, we set for $k\in\ZZ$
$$
\chi[k]\;:=\;\chi\otimes\left(\mathcal N^{\otimes k}\right).
$$
If $\chi$ is algebraic, then so is $\chi[k]$.

Fix a non-trivial continuous character $\psi:F\backslash\Adeles_F\to\CC^\times$, and all Gau\ss{} sums we consider are understood with respect to $\psi$, and normalized in such a way that when $\chi=\chi'\otimes|\cdot|_{\Adeles_F}^{k(\chi)}$, with $\chi'$ of finite order, then
$$
G(\chi)\;:=\;
G(\chi')\;:=\;
\sum_{x{\rm mod}\mathfrak{f}_{\chi'}}
\chi'\left(\frac{x}{f_{\chi'}}\right)
\psi\left(\frac{x}{f_{\chi'}}\right),
$$
where $\mathfrak{f}_{\chi'}=\mathcal O_F\cdot f_{\chi'}$ is the conductor of the finite order character $\chi'.$

\subsection{Rational models of compact groups}\label{sec:compactmodels}

We fix for each $m$ a model $K_m\subseteq G_m$ of the standard maximal compact subgroup of $G_m(\RR)$, which is admissible in the sense of section 6 in \cite{januszewski2017} for the quadratic extension $\QQ_K(\sqrt{-1})/\QQ_K$ for a field of definition $\QQ_K\subseteq\RR$ of $K_m,$ which we assume to be a number field and independent of $m.$ In particular, $\QQ_K$ comes with a fixed embedding $\QQ_K\to\RR$ and $K_m$ splits over $E_K:=\QQ_K(\sqrt{-1}).$ The choices in section 6.4 of loc.\ cit.\ show that if $F$ is totally real or a CM field, by choosing $K_m$ quasi-split all all finite primes $\neq 2,$ we can arrange for $\QQ_K=\QQ.$ Recall
\begin{equation}
K_m(\RR)\;=\;
\Oo(m,\RR)^{r_F^\RR}\times
\U(m,\RR)^{r_F^\CC}.
\label{eq:klocaldecomposition}
\end{equation}
Corresponding to $K_m,$ we have a Cartan involution $\theta_m:G_m(\RR)\to G_m(\RR)$, which is defined over $\QQ_K$. Therefore we may and do consider $\theta_m$ as an automorphism of $G_m$ over $\QQ_K$. We write $\tau_K\in\Gal(E_K/\QQ)$ for the unique non-trivial automorphism, which we simply call {\em complex conjugation}.

We remark that we have a natural isomorphism
$$
\pi_0(K_m)\;=\;\pi_0(K_m(\RR)),
$$
where $K_m(\RR)$ on the right hand side is considered as a real Lie group.

\subsection{Rational models of pairs}

We let $G$ denote a connected reductive group over $\QQ$, $K\subseteq G$ a closed reductive subgroup, and write $\lieg$, $\liek$, $\lieg_n$, ... for the Lie algebras of $G$, $K$, $G_n$, ... respectively. All these Lie algebras are defined over $\QQ$. To stress the base field $E/\QQ$ under consideration, we write $\lieg_E\;:=\;\lieg\otimes_\QQ E,$ and similarly $G_E\;:=\;G\times_\QQ E$ for the base change of $G\to\Spec\QQ$ to $G_E\to\Spec E$. Then $(\lieg_E,K_E)$ is a reductive pair over $E$ in the sense of section 1.4 of \cite{januszewski2017}. For the sake of readability, introduce the notation $(\lieg,K)_E\;:=\;(\lieg_E,K_E),$ that we also apply to more general pairs, which we will discuss in more detail in section \ref{sec:hcmodules} below.

We have a natural isomorphism between the group of quasi-characters $\Hom(G_n(\RR),\CC^\times)$ and the group of one-dimensional $(\lieg_m,K_m)_\CC$-modules, sending a quasi-character $\chi$ to its derivative $d\chi: \lieg_{m,\CC}\to\CC,$ and to the complexification of its restriction to $K_m(\RR)$, $\chi|_{K_m(\RR),\CC}: K_m(\CC)\to\CC^\times.$ For the sake of notational simplicity we write $\omega_s^\delta$ also for the image of $\omega_s^\delta$ under this correspondence. Then, since $\mathcal N$ is defined over $\QQ$, and $\sgn_\infty^\delta$, as a rational character of $K_m$, is defined over $\QQ_K$ as well, we see with \eqref{eq:algebraicarchimedeannorm}, that for each $s\in\ZZ$ the algebraic quasi-character
$$
\omega_s^\delta\;=\;\sgn_\infty^\delta\otimes(\sgn_\infty\otimes\eta)^s,
$$
considered as a one-dimensional $(\lieg_m,K_m)$-module, is defined over $\QQ_K$ as well. We write $X^{\rm alg}(G_n(\RR))$ for the group of algebraic quasi-characters of $G_m(\RR)$ or $(\lieg_n,K_n)$, equivalently.

\subsection{Finite-dimensional representations}

Since $G$ is quasi-split we may find a Borel $P\subseteq G$ defined over $\QQ$. We choose a maximal torus $T\subseteq P$ over $\QQ$ and denote by $P^-$ the opposite of $P$. We let $X(T)$ denote the characters of $T$ defined over $\overline{\QQ}$ or $\CC$, which amounts to the same, and let $X_E(T)$ denote the subgroup of characters defined over $E$. Then the absolute Galois group $\Gal(\overline{\QQ}/\QQ)$ and $\Aut(\CC/\QQ)$ act naturally on $X(T)$ from the right via the rule
\begin{equation}
\mu^\tau(t)\;=\;\mu\left(t^{\tau^{-1}}\right)^{\tau},
\label{eq:galoismu}
\end{equation}
for $\mu\in X(T)$, $t\in T(\CC)$ and $\tau\in\Aut(\CC/\QQ)$. In general there is a second action on $X(T)$, introduced Borel and Tits in \cite{boreltits1965,boreltits1972}, denoted ${}_{\Delta}\tau(\mu)$ in loc.\ cit. The latter action maps dominant weights to dominant weights and reflects the action of $\Gal(\overline{\QQ}/\QQ)$ and $\Aut(\CC/\QQ)$ on the category of rational representations. However, since $G$ is quasi-split these two actions agree in our case. In particular if $M_\mu$ denotes the irreducible rational $G$-module of highest weight $\mu\in X(T)$ over $\CC$ (or $\overline{\QQ}$), then the Galois twisted module
\begin{equation}
(M_\mu)^\tau\;:=\;M_\mu\otimes_{\CC,\tau^{-1}} \CC
\label{eq:galoistwistedmodule}
\end{equation}
is irreducible of highest weight $\mu^\tau$. Furthermore, it is easy to see that $M_\mu$ is defined over the field of rationality $\QQ(\mu)$ of $\mu$. For all these assertions we refer to section 12 in \cite{boreltits1972}.

\subsection{Harish-Chandra modules}\label{sec:hcmodules}

If $(V,\rho)$ is a Casselman-Wallach representation of $G(\RR)$, we write $V^{(K)}$ for the subspace of $K(\RR)$-finite vectors. This is a finitely generated admissible $(\lieg,K)_\CC$-module, and is irreducible if and only if $V$ is (topologically) irreducible. Indeed, the categories of finite length $(\lieg,K)_\CC$-modules is equivalent to the category of Casselman-Wallach representations of $G(\RR)$, i.e. finitely generated, admissible, Fr\'echet representations of moderate growth. That this correspondence fails for non-admissible modules is at the heart of the period relation problem under investigation.

We give a sketch of the theory of Harish-Chandra modules over arbitrary fields of characteristic $0$. The relevant theory of cohomological induction over general fields will be taken up in section \ref{sec:cohomologicalinduction}. For the fundamental theory of modules and cohomological induction over arbitrary fields, as well as the rationality results we need, we refer to \cite{januszewski2017}.

A pair over a field $E/\QQ$ consists of an $E$-Lie algebra $\liea_E$ and a reductive algebraic group $B_E$ over $E$, together with an inclusion
$$
\iota_E:\quad \Lie(B_E)\to\liea_E,
$$
of Lie algebras and an action of $B_E$ on $\liea_E$, extending the action of $B_E$ on $\Lie(B_E)$, whose derivative is the adjoint action of $\Lie(B_E)$ on $\liea_E$, the latter action being induced by $\iota_E$.

Then an $(\liea,B)_E$-module $X_E$ consists of an $E$-vector space $X_E$ together with compatible actions of $\liea_E$ and $B_E$. Here we implicitly assume that $X_E$ is a {\em rational} $B_E$-module, which amounts to saying that it is a direct sum of finite-dimensional rational representations of $B_E$. Then this rational action induces an action of $\Lie(B_E)$ on $X_E$, and we assume the action of $\liea_E$ to be an extension of this action. Furthermore, we assume the given adjoint action of $B_E$ on $\liea_E$ and the given action on $X_E$ to be compatible with the action of $\liea_E$ in the usual sense.

Then the category $\mathcal C(\liea,B)_E$ of $(\liea,B)_E$-modules (over $E$) is an abelian category, even an $E$-linear tensor category. For each extension of fields $E'/E$ we have natural base change functors
$$
-\otimes_E E':\quad \mathcal C(\liea,B)_E\to \mathcal C(\liea,B)_{E'},
$$
sending an $E$-rational module $X_E$ to the $E'$-rational module
$$
X_{E'}\;:=\;X_E\otimes_E E',
$$
which is an $(\liea,B)_{E'}$-module. A fundamental property of the base change functor is
$$
\Hom_{\liea,B}(X_E,Y_E)\otimes_E E'\;=\;
\Hom_{\liea,B}(X_{E'},Y_{E'})
$$
for $X_E,Y_E\in\mathcal C(\liea,B)_E$ under suitable finiteness conditions, cf.\ Proposition 1.1 in \cite{januszewski2017}. This observation enables us to control the rationality of functionals.

In this context, remark that the algebraic quasi-characters in $X^{\rm alg}(G_m(\RR))$, since characters of the pair $(\lieg_m,K_m)$, are all defined over $\QQ_K.$

We remark furthermore that if $(\liea,B)$ is a pair over $\QQ$, then the Galois action on modules defined in \eqref{eq:galoistwistedmodule} extends naturally to arbitrary $(\liea,B)$-modules.

\section{Cohomologically induced modules over $\QQ$}\label{sec:cohomologicalinduction}

In this section we discuss the fields of definition of tensor products of tempered cohomologically induced standard modules on products of $\GL_n$. We first recall classical results which are fundamental to our subsequent treatment.

\subsection{Representations of orthogonal groups}\label{sec:repon}

For $n\geq 1$, we choose a maximal torus $T\subseteq \SO(2n+1)$. Then the root system $\Delta\subseteq X^*(T)$ of $\SO(2n+1)$ is of type $B_n$, and we identify
$$
X^*(T)\otimes_\ZZ\RR=\RR^n,
$$
by means of the standard orthonormal basis $e_1,\dots, e_n$ of $\RR^n$. We assume that in our notation the root system is given by
$$
\Delta\;=\;\{\pm e_i\pm e_j\mid i<j\}\;\cup\;\{\pm e_i\},
$$
our choice of simple roots in this notation being $e_1-e_2,e_2-e_3\dots,e_{n-1}-e_n,e_n.$ The orthogonal group $\Oo(2n+1)$ only has inner automorphisms and is a direct product
$$
\Oo(2n+1)\;=\;\SO(2n+1)\times\{\pm 1\},
$$
and thus irreducible $\Oo(2n+1)$-modules $W^\pm(\lambda)$ are indexed by a sign $\pm$ and analytically integral dominant weights
$$
\lambda\;=\;\lambda_1 e_1+\cdots+\lambda_n e_n,\quad\lambda_1,\dots,\lambda_n\in\ZZ,
$$
for $\SO(2n+1)$, satisfying the dominance condition
$$
\lambda_1\geq \lambda_2\geq\cdots\geq
\lambda_n\geq 0.
$$
In the even case, consider for $n\geq 2$ the root system
$$
\Delta\subseteq X^*(T)\subseteq X^*(T)\otimes_\ZZ\RR=\RR^n,
$$
of $\SO(2n)$ for a maximal torus $T\subseteq\SO(2n)$, with the same standard basis as above. It is of type $D_n$, and $\Delta\;=\;\{\pm e_i\pm e_j\mid i<j\}.$ Fix the simple roots as $
e_1-e_2,e_2-e_3\dots,e_{n-1}-e_n,e_{n-1}+e_n.$

The outer automorphism group of $\Oo(2n)$ is of order two and the orthogonal group $\Oo(2n)$ is the unique non-split semidirect product
$$
\Oo(2n)\;=\;\SO(2n)\rtimes\{\pm 1\}.
$$
By classical Mackey Theory there are two distinct cases for the structure of irreducible rational $\Oo(2n)$-modules. Given an analytically integral dominant weight $\lambda$ of $\SO(2n)$, i.e.\ satisfying
$$
\lambda_1\geq \lambda_2\geq\cdots\geq
|\lambda_n|,
$$
the induced representation
$$
W(\lambda)\;=\;\ind_{\SO(2n)}^{\Oo(2n)}V(\lambda)
$$
is irreducible whenever $\lambda_n\neq 0$, and its isomorphism class depends only on the tuple $(\lambda_1, \lambda_2,\dots,|\lambda_n|)$. In the case $\lambda_n=0$, the module $\ind_{\SO(2n)}^{\Oo(2n)}V(\lambda)$ decomposes into two non-isomorphic irreducible $\Oo(2n)$-representations $W^\pm(\lambda)$, which differ by a sign twist.

The following classical result will be relevant for our descent argument in the proof of Theorem \ref{thm:connectedbottomlayerdefinition}.

\begin{proposition}\label{prop:realorthogonalrepresentations}
For every $n\geq 1$ and every analytically integral dominant weight $\lambda$ for $\SO(2n+1)$ and every sign $\pm$, the irreducible complex $\Oo(2n+1)$-module $W^{\pm}(\lambda)$ is self-dual and real, i.e.\ defined over $\RR$, hence so is the underlying irreducible $\SO(2n+1)$-module.\\
Likewise, for every analytically integral dominant weight $\lambda$ for $\SO(2n)$ with $\lambda_n\neq 0$ the irreducible complex $\Oo(2n)$-module $W(\lambda)$ is always real. As a complex $\SO(2n)$-module it decomposes into a direct sum
$$
W(\lambda)\;=\;V(\lambda)\oplus V(\tilde{\lambda})
$$
of two irreducible complex $\SO(2n)$-modules. $V(\lambda)$ (resp.\ $V(\tilde{\lambda})$) is real if and only if it is self-dual. This is so if and only if $n$ is even.
\end{proposition}

\begin{proof}
By Proposition 6.10 in \cite{januszewski2017} every self-dual irreducible complex $\SO(2n)$- and $\SO(2n+1)$-module is real. Hence the claim for $\Oo(2n+1)$ and $\Oo(2n)$ follows with Lemma 6.3 of loc.\ cit.

The only detail which is not explicitly covered by this reasoning is the last statement that for a dominant weight $\lambda$ with $\lambda_n\neq 0$ the module $V(\lambda)$ is self-dual only for even $n$. If $n$ is even, the action of $-w_0$ on the weight space is trivial. For odd $n$ we have for the weight $\lambda$ under consideration
$$
-w_0(\lambda)\;=\;(\lambda_1,\dots,\lambda_{n-1},-\lambda_{n})\;=\;\tilde{\lambda},
$$
and since $\lambda_n\neq 0$, the claim follows.
\end{proof}

\subsection{Representations of unitary groups}

For unitary groups we content us to import the well known

\begin{proposition}[Special case of Proposition 6.14 in \cite{januszewski2017}]\label{prop:realunitaryrepresentations}
Let $V$ be an irreducible complex representation of $\U(n).$ Then $V$ is defined over $\RR$ if and only if it is self-dual.
\end{proposition}

This proposition may be proved by elementary means using root systems and Tits' Theorem was well (cf.\ the proof of Proposition 6.14 in \cite{januszewski2017}).

\subsection{Induction data}

We fix a finite sequence $n_1,n_2,\dots,n_r$ of $r\geq 1$ positive integers and set
$$
G\;:=\;\res_{F/\QQ}\GL_{n_1}\times\GL_{n_2}\times\cdots\times\GL_{n_r},
$$
where $F/\QQ$ is a number field as before. Recall that $\lieg$ denotes the rational Lie algebra of $G$ over $\QQ$, and we fixed the standard choice
\begin{equation}
K\;:=\;K_{n_1}\times K_{n_2}\times\cdots\times K_{n_r}\;\subseteq\;G
\label{eq:kproduct}
\end{equation}
of a $\QQ_K$-rational model of a standard maximal compact subgroup of $G(\RR)$ as in section \ref{sec:compactmodels}. Along with $K$ comes a Cartan involution $\theta$ of $G$ and $\lieg$, which is defined over $\QQ_K$ as well.

The product decomposition of $G$ induces a product decomposition
\begin{equation}
\lieg\;=\;\lieg_1\times\lieg_2\times\cdots\times\lieg_r,
\label{eq:gproduct}
\end{equation}
and similarly for the Lie algebra $\liek$ of $K$.

Recall that we fixed an imaginary quadratic extension $E_K/\QQ$ where $K$ is split and that $K$ is $E_K/\QQ_K$-admissible in the sense of \cite{januszewski2017}. In later parts of the text we will exploit that $E_K$ is given by $\QQ_K(\sqrt{-1})$, but in this section a general purely imaginary quadratic extension $E_K/\QQ_K$ where $K$ splits is just as good. We fix a $\theta$-stable Borel subalgebra $\lieq\subseteq\lieg_{E_K}$, where $\theta$-stability means that
\begin{equation}
\theta(\lieq)\;=\;\lieq.
\label{eq:thetastability}
\end{equation}
Later in the text we will assume $\lieq$ to be transversal to a specific diagonally embedded Lie algebra $\lieh$, but for now any $\theta$-stable $\lieq$ which factors as a product
$$
\lieq\;=\;\lieq_1\times\lieq_2\times\cdots\times\lieq_r,
$$
according to \eqref{eq:gproduct} is sufficient for our purpose.

By \eqref{eq:thetastability}, the Lie algebra
$$
\lieq\cap\liek_{E_K}\subseteq\liek_{E_K}
$$
is a Borel subalgebra of $\liek_{E_K}$. Again by \eqref{eq:thetastability}, complex conjugation $\tau_K\in\Gal(E_K/\QQ)$ maps $\lieq$ to its opposite
$$
\overline{\lieq}\;:=\lieq^{\tau_K}\;=\;\lieq^-,
$$
if we consider the unique $\theta$-stable and $\QQ_K$-rational Levi factor given explicitly by $\liec\;:=\;\lieq\cap\overline{\lieq}.$ Again, $\liec$ factors as a product
$$
\liec\;=\;\liec_1\times\liec_2\times\cdots\times\liec_r,
$$
accordingly. Let $\lieu\subseteq\lieq$ denote the nilpotent radical of $\lieq$. Then $\lieq_{E_K}\cap\liek_{E_K}$ has the Levi decomposition
$$
\lieq_{E_K}\cap\liek_{E_K}\;=\;(\liec_{E_K}\cap\liek_{E_K})\oplus (\lieu_{E_K}\cap\liek_{E_K}).
$$

We fix the compact factor $C$ of the Levi pair associated to $\lieq$ as follows. For each index $1\leq i\leq r$ we set
$$
C_i\;:=\;Z_{K_i}(\liec_i).
$$
If $n$ is even, then $C_i$ is contained in $K_i^0$; if $n$ is odd, it meets every connected component of $K_i$. Finally
$$
C\;:=\;C_1\times C_2\times\cdots\times C_r.
$$
Then $C$ is defined over $\QQ_K,$ and $(\lieq,C)_{E_K}$ is what we call a $\theta$-stable {\em parabolic pair} over $E_K$.

We consider rational models of cohomologically induced $(\lieg_\CC,K_\CC)$-standard-modules $A_{\lieq}(\mu)_\CC$ arising as follows.

Departing from a rational absolutely irreducible representation $M(\mu)$ of $G$ with highest weight $\mu$ with respect to $\lieq_\CC,$ we have a one-dimensional highest weight space $H^0(\lieu;M(\mu))$ inside of $M(\mu)$. Since $C$ normalizes $\lieq$, the inclusion $C\to G$ induces an action of $C$ on $M(\mu)$ under which $H^0(\lieu;M(\mu))$ is stable. Hence, we obtain a character $\mu$ of the pair $(\liec,C)_\CC$, or equivalently of the parablic pair $(\lieq,C)_\CC$ with trivial action of the radical. This agrees with the natural action of the pair $(\lieq,C)_\CC$ on $H^0(\lieu;M(\mu))$. All these data are indeed defined over $E_K(\mu)=E_K\cdot\QQ(\mu)$, where $\QQ(\mu)$ is the field of rationality of $M(\mu)$, once we fix an inclusion
\begin{equation}
\iota:\quad E_K(\mu)\to\CC,
\label{eq:complexembedding}
\end{equation}
extending the given emnbedding $\QQ_K\to\RR$. We assume this to be the case in all what follows.

In \cite{januszewski2017}, the author introduced Zuckerman functors over arbitrary fields of characteristic $0$. Put
$$
S_\lieq\;:=\;\dim(\lieu_{E_K}\cap\liek_{E_K}),
$$
and denote by $R^q\Gamma_{\lieg,C}^{\lieg,K}$ the $E_K(\mu)$-rational $q$-th right derived Zuckerman functor for the inclusion of pairs $(\lieg,C)\to(\lieg,K)$, as introduced in \cite{januszewski2017}. The construction of Zuckerman functors in loc.\ cit.\ generalizes Zuckerman's original construction discussed in \cite{book_vogan1981}. For the fundamental properties of Zuckerman's cohomological induction over $\CC,$ the reader is refered to the excellent monograph \cite{book_knappvogan1995}.

As in the classical case, the functor $\Gamma_{\lieg,CK^0}^{\lieg,K}$ is exact and given by induction along the inclusion $CK^0\to K$. Now $\Gamma_{\lieg,C}^{\lieg,CK^0}$, being right adjoint to an exact forgetful functor, maps injectives to injectives. Therefore, the Grothendieck spectral sequence associated to the composition
$$
\Gamma_{\lieg,C}^{\lieg,K}\;=\;\Gamma_{\lieg,CK^0}^{\lieg,K}\circ \Gamma_{\lieg,C}^{\lieg,CK^0}
$$
induces in each degree $q$ a natural isomorphism of functors
\begin{equation}
R^q\Gamma_{\lieg,C}^{\lieg,K}\;=\;\Gamma_{\lieg,CK^0}^{\lieg,K}\circ R^q\Gamma_{\lieg,C}^{\lieg,CK^0}
\label{eq:gammacomposition}
\end{equation}
Using the author's rational construction, we have the $E_K(\mu)$-rational module
$$
A_\lieq(\mu)\;:=\;R^{S_\lieq}\Gamma_{\lieg,C}^{\lieg,K}(\Hom_{\lieq}(U(\lieg_{E_K(\mu)}),\mu\otimes\bigwedge^{\dim\lieu}\lieu_{E_K(\mu)})^{(C)}),
$$
where the superscript $(\cdot)^{(C)}$ denotes the subspace of $C$-finite vectors as before. Then, since rational Zuckerman functors commute with base change in the case at hand (Theorem 2.5 of \cite{januszewski2017}), we have an isomorphism
$$
A_{\lieq}(\mu)\otimes\CC\;\cong\;A_{\lieq}(\mu)_\CC,
$$
which we use to identify the right hand side with the left hand side, justifying this abuse of notation.

For the study of periods we introduce the intermediate module
$$
A_\lieq^\circ(\mu)\;:=\;R^{S_\lieq}\Gamma_{\lieg,C}^{\lieg,CK^0}(\Hom_{\lieq}(U(\lieg_{E_K(\mu)}),\mu\otimes\bigwedge^{\dim\lieu}\lieu_{E_K(\mu)})^{(C)}).
$$
By \eqref{eq:gammacomposition}, its relation to $A_\lieq(\mu)$ is
\begin{equation}
A_\lieq(\mu)_{E_K}\;=\;\Gamma_{\lieg,CK^0}^{\lieg,K}(A_\lieq^\circ(\mu)_{E_K}),
\label{eq:aqintermediateinduction}
\end{equation}
and by Frobenius reciprocity, we may consider $A_\lieq^\circ(\mu)$ as an irreducible $(\lieg,CK^0)_{E_K(\mu)}$-submodule of $A_\lieq(\mu)_{E_K(\mu)}$, and the $K$-translates of the former generate the latter as an $E_K(\mu)$-vector space.

To be more precise, $A_\lieq(\mu)$ decomposes into a direct sum
\begin{equation}
A_\lieq(\mu)\;=\;\bigoplus_{\varepsilon}\varepsilon A_\lieq^\circ(\mu),
\label{eq:standardmoduledecomposition}
\end{equation}
where $\varepsilon$ runs through a suitable system of representatives of $K/CK^0.$ The author showed in Theorem 7.3 of \cite{januszewski2017} that $A_{\lieq}(\mu)$ is actually defined over $\QQ_K(\mu)$. In general, this is no more true for $A_\lieq^\circ(\mu)$ in the presence of real places. We will determine its field of definition in Theorem \ref{thm:connectedbottomlayerdefinition} below. This result is crucial for our applications to special values.

For the sake of completeness, we remark that the natural embedding $\QQ_K(\mu)\to E_K(\mu)$ induces with \eqref{eq:complexembedding} an embedding $\QQ_K(\mu)\to\CC$, which we understand fixed in the sequel. In particular, \eqref{eq:moduleisomorphism} retains its meaning also for the model of $A_{\lieq}(\mu)$ over $\QQ_K(\mu)$, the latter being unique by virtue of Proposition 3.4 of loc.\ cit.

\subsection{The bottom layer}

This section recalls fundamental facts about the structure of bottom layer of the standard modules of interest as a rational representation of $K$ and $K^0$, respectively. The understanding of rationality properties of the bottom layer is key to the understanding of the rationality properties of the intermediate modules $A_{\lieq}^\circ(\mu)$.

The bottom layer is defined as
$$
B_{\lieq}(\mu)_{E_K(\mu)}\;:=\;
A_{\lieq\cap\liek}(\mu|_{C}\otimes\!\!\!\!\!\bigwedge^{\dim\lieu/\lieu\cap\liek}\!\!\!\!\!\!\!(\lieu/\lieu\cap\liek))_{E_K(\mu)}
\;\subseteq\;
A_{\lieq}(\mu)_{E_K(\mu)},
$$
where
$$
A_{\lieq\cap\liek}(\cdot)\;
\;:=\;
R^{S_\lieq}\Gamma_{\liek,C}^{\liek,K}
\left((\cdot)\otimes\!\!\!\!\!\bigwedge^{\dim\lieu\cap\liek}\!\!\!\!\!\lieu\cap\liek\right).
$$
The descent argument in the proof of Theorem 7.3 in \cite{januszewski2017} tells us that the bottom layer itself is defined over $\QQ_K(\mu)$. Again its model over $\QQ_K(\mu)$ is unique and we denote it the same.

As in the non-compact case, introduce the intermediate module
$$
B_{\lieq}^\circ(\mu|_{C})_{E_K(\mu)}:=A_{\lieq\cap\liek}^\circ(\mu|_{C})_{E_K(\mu)}
:=R^{S_\lieq}\Gamma_{\liek,C}^{\liek,CK^\circ}(\mu|_{C}\otimes\bigwedge^{\dim\lieu}\lieu_{E_K(\mu)})
\subseteq A_{\lieq}^\circ(\mu)_{E_K(\mu)}.
$$
This is an irreducible representation of $K^0$ of highest weight
\begin{equation}
\mu_K\;:=\;\mu|_{C^0}+2\rho(\lieu/\lieu\cap\liek),
\label{eq:ktypeweight}
\end{equation}
where
$$
\rho(\lieu/\lieu\cap\liek)\;=\;\frac{1}{2}\sum_{\alpha\in\Delta(\lieu/\lieu\cap\liek,\liec)}\alpha,
$$
is the half sum of weights in $\lieu/\lieu\cap\liek$. Again,
\begin{equation}
B_\lieq(\mu)_{E_K(\mu)}\;=\;\Gamma_{\lieg,CK^0}^{\lieg,K}(B_\lieq^\circ(\mu)_{E_K(\mu)}).
\label{eq:bqintermediateinduction}
\end{equation}

\subsection{Structure of the bottom layer}

We assume that $n_1,\dots,n_s$ are even and $n_{s+1},\dots,n_r$ are odd, for some $0\leq s\leq r$. By the preceeding discussion and the transitivity principle \eqref{eq:bqintermediateinduction}, the representation $B_{\lieq}(\mu)_{\CC}$ decomposes as a complex representation of
\begin{eqnarray*}
K^0(\RR)&=&
\left(\SO(n_1,\RR)^{r_F^\RR}\times\U(n_1,\RR)^{r_F^\CC}\right)\times
\left(\SO(n_2,\RR)^{r_F^\RR}\times\U(n_1,\RR)^{r_F^\CC}\right)
\times\\
&&\cdots\times\left(\SO(n_r,\RR)^{r_F^\RR}\times\U(n_r,\RR)^{r_F^\CC}\right),
\end{eqnarray*}
according to Proposition \eqref{prop:realorthogonalrepresentations} into a direct sum
\begin{equation}
B_{\lieq}(\mu)_{\CC}\;=\;
\bigotimes_{i=1}^s
\left(
\bigotimes_{v\;\text{real}}\left(B_{\mu_{v,i}}^{\SO}\oplus B_{\tilde{\mu}_{v,i}}^{\SO}\right)
\otimes
\bigotimes_{v\;\text{cplx}}B_{\mu_{v}}^{\U}
\right)
\otimes
\bigotimes_{i=s+1}^r
\left(
\bigotimes_{v\;\text{real}}B_{\mu_{v,i}}^{\SO}
\otimes
\bigotimes_{v\;\text{cplx}}B_{\mu_{v}}^{\U}
\right),
\label{eq:bottomlayerdecomposition}
\end{equation}
where $B_{\mu_{v,i}}^{\SO}$ is the complex irreducible representation of $\SO_{n_i}(\RR)$ of highest weight $\mu_{v,i}$, $B_{\mu_{v,i}}^{\U}$ is the complex irreducible representation of $\U(n_i,\RR)$ of highest weight $\mu_{v,i}$, and finally
$$
\mu_K\;=\;\otimes_{i=1}^r\otimes_{v\mid\infty}\;\mu_{i,v}
$$
is the natural factorization of $\mu_K$ according to \eqref{eq:kproduct} and the component-wise factorization into local components. Similarly, we have
\begin{equation}
  B_{\lieq}^\circ(\mu)_{\CC}\;=\;
  \bigotimes_{i=1}^r
  \left(
  \bigotimes_{v\;\text{real}}B_{\mu_{v,i}}^{\SO}
  \otimes
  \bigotimes_{v\;\text{cplx}}B_{\mu_{v}}^{\U}
  \right),
\label{eq:bottomlayerzerodecomposition}
\end{equation}
again as a complex representation of $K^0(\RR)$.

We remark that since the representation $B_{\lieq}^\circ(\mu)_{\CC}$ is defined over $E_K(\mu)$, the summands in the direct sum decompositions \eqref{eq:standardmoduledecomposition} and \eqref{eq:bottomlayerdecomposition} are defined over $E_K(\mu)$ as well. However, we need to decide when these summands descend to $\QQ_K(\mu).$

\begin{lemma}\label{lem:connectedbottomlayerdefinition}
The representation $B_{\lieq}^\circ(0)_{\CC}$ of $K^0$ and the $(\lieg,K^0)$-module $A_{\lieq}^\circ(0)_\CC$ are defined over $E_K$. The module $B_{\lieq}^\circ(0)_{\CC}$ resp.\ $A_{\lieq}^\circ(0)_\CC$ is defined over $\QQ_K$ if and only if for every $1\leq i\leq r$ we have
\begin{equation}
2\mid n_i\;\Longrightarrow\;4\mid n_i.
\label{eq:paritycondition}
\end{equation}
\end{lemma}

\begin{proof}
By the descent argument (cf.\ proof of Theorem 7.3 in \cite{januszewski2017}), the statement of the lemma about the bottom layer is equivalent to the statement that the complex representation $B_{\lieq}^\circ(0)_{\CC}$ is real if and only if \eqref{eq:paritycondition} is satisfied.

For $B_{\lieq}^\circ(0)_{\CC}$ to be real it is necessary that $B_{\lieq}^\circ(0)_{\CC}$ be self-dual. By Proposition \ref{prop:realorthogonalrepresentations}, this cannot be the case if there is an index $1\leq i\leq r$ for which $n_i\equiv 2\pmod 4$. Moreover, self-duality is also a sufficient condition, again by Proposition \ref{prop:realorthogonalrepresentations}. Hence the claim for $B_{\lieq}^\circ(0)_{\CC}$ follows.

As in the proof of Theorem 7.3 of \cite{januszewski2017} we may appeal to Proposition 5.8 of loc.\ cit.\ to conclude that $A_{\lieq}^\circ(0)_{\CC}$ is defined over $\QQ_K$. Alternatively we may argue that, as a submodule
$$
A_{\lieq}^\circ(0)_{E_K}\;\subseteq\;A_{\lieq}(0)_{E_K}
$$
of a module which is defined over $\QQ_K$ (Theorem 7.3 of loc.\ cit.), it is defined over $\QQ$ if and only if it is invariant under the action of $\Gal(E_K/\QQ)$, and this is so if and only if $B_{\lieq}^\circ(0)_{E_K}$  is invariant under the action of $\Gal(E_K/\QQ)$, since each irreducible submodule of $A_{\lieq}(0)_{E_K}$ is uniquely determined by the unique minimal $K^0$-type it contains.
\end{proof}

\begin{theorem}\label{thm:connectedbottomlayerdefinition}
The $K^0$-module $B_{\lieq}^\circ(\lambda)_{\CC}$ and the $(\lieg,K^0)$-module $A_{\lieq}^\circ(\lambda)_\CC$ are defined over $E_K(\mu)$. If $\sqrt{-1}\not\in\QQ_K(\mu)$, then $B_{\lieq}^\circ(\lambda)_{\CC}$ resp.\ $A_{\lieq}^\circ(\lambda)_\CC$ is defined over $\QQ_K(\mu)$ if and only if for every index $1\leq i\leq r$,
$$
2\mid n_i\;\Longrightarrow\;4\mid n_i.
$$
\end{theorem}

We remark that the same rationality statement applies to the summands in the direct sum decomposition \eqref{eq:bottomlayerdecomposition} and in the decomposition of $A_\lieq(\mu)$ into absolutely irreducible $(\lieg,K^0)$-modules.

\begin{proof}
We follow the argument of the proof of Proposition 7.1 in \cite{januszewski2017}. By Lemma \ref{lem:connectedbottomlayerdefinition}, the $(\lieg,K^0)_{E_K}$-submodule
$$
A_{\lieq}^\circ(0)_{E_K}\;\subseteq\;A_{\lieq}(0)_{E_K}
$$
is defined over $\QQ_K$ if and only if \eqref{eq:paritycondition} is satisfied for all $1\leq i\leq r$. Assume $\sqrt{-1}\not\in\QQ_K(\mu),$ otherwise there is nothing to prove. Consider for each dominant integral weight $\nu$ the translation functor
$$
{\mathcal T}_{\mu}^{\nu}:\quad X\mapsto \pi_{\nu}(X\otimes M(\mu)),
$$
where $\pi_{\nu}$ denotes the projection onto the submodule on which the center $Z(\lieg)\subseteq U(\lieg)$ acts as in $M(\nu)$. Then we have a commutative square
$$
\begin{CD}
A_{\lieq}^\circ(\mu)_{E_K(\mu)}@>\sim>>{\mathcal T}_{\mu}^{\mu}(A_{\lieq}^\circ(0)_{E_K(\mu)})\\
@VVV @VVV\\
A_{\lieq}(\mu)_{E_K(\mu)}@>\sim>>{\mathcal T}_{\mu}^{\mu}(A_{\lieq}(0)_{E_K(\mu)})\\
\end{CD}
$$
and likewise for the contragredient weight $\mu^\vee$,
$$
\begin{CD}
A_{\lieq}^\circ(0)_{E_K(\mu)}@>\sim>>{\mathcal T}_{\mu^\vee}^{0}(A_{\lieq}^\circ(\mu)_{E_K(\mu)})\\
@VVV @VVV\\
A_{\lieq}(0)_{E_K(\mu)}@>\sim>>{\mathcal T}_{\mu^\vee}^{0}(A_{\lieq}(\mu)_{E_K(\mu)})\\
\end{CD}
$$
Therefore, $A_{\lieq}^\circ(\mu)_{E_K(\mu)}$ is defined over $\QQ_K(\mu)$ if and only if $A_{\lieq}^\circ(0)_{E_K(\mu)}$ has a model over $\QQ_K(\mu)$.

In light of Lemma \ref{lem:connectedbottomlayerdefinition}, this proves the implication
$$
\left(\forall i:
2\mid n_i\Longrightarrow 4\mid n_i\right)
\;\Longrightarrow\;A_{\lieq}^\circ(\mu)_{E_K(\mu)}\;\text{is defined over}\;\QQ_K(\mu).
$$

To see the other implication, it remains to show by the preceeding discussion that if $A_{\lieq}^\circ(0)_{E_K(\mu)}$, being a translate of $A_{\lieq}^\circ(\mu)_{E_K(\mu)}$, has a model over $\QQ_K(\mu)$, then $A_\lieq^0(0)_{E_K}$ has already a model over $\QQ_K$.

Due to our assumption that $\sqrt{-1}\not\in\QQ_K(\mu)$, the extensions $E_K$ and $\QQ_K(\mu)$ are linearly disjoint over $\QQ$, and restriction induces an isomorphism of Galois groups
\begin{equation}
\begin{CD}
\Gal(E_K(\mu)/\QQ_K(\mu))@>\sim>>\Gal(E_K/\QQ).
\end{CD}
\label{eq:galoisisomorphism}
\end{equation}
Therefore, if $A_{\lieq}^\circ(0)_{E_K(\mu)}$ has a model over $\QQ_K(\mu)$, complex conjugation $\tau_K\in\Gal(E_K(\mu)/\QQ_K(\mu))$ stabilizes $A_{\lieq}^\circ(0)_{E_K(\mu)}$ as a submodule of $A_{\lieq}(0)_{E_K(\mu)}$. By \eqref{eq:galoisisomorphism} it therefore stabilizes the subspace $A_{\lieq}^\circ(0)_{E_K}$, which, by Galois descent for vector spaces, therefore is defined over $\QQ_K$. This implies, with Lemma \ref{lem:connectedbottomlayerdefinition},
\begin{equation}
\forall i:\quad
2\mid n_i\;\Longrightarrow\; 4\mid n_i.
\label{eq:divisibilityrelation}
\end{equation}
This concludes the proof of the statement about standard modules.

To determine the fields of definition of the minimal $K^0$-types, we observe that
$$
B_{\lieq}^\circ(\mu)_{E_K(\mu)}\;=\;B_{\lieq}(\mu)_{E_K(\mu)}\cap A_{\lieq}^\circ(\mu)_{E_K(\mu)},
$$
which shows that $B_{\lieq}^\circ(\mu)_{E_K(\mu)}$ is defined over $\QQ_K(\mu)$ whenever \eqref{eq:divisibilityrelation} is satisfied.

Assume conversely that $B_{\lieq}^\circ(\mu)_{E_K(\mu)}$ is defined over $\QQ_K(\mu)$. As in the end of the proof of Lemma \ref{lem:connectedbottomlayerdefinition} we may argue via Galois descent for vector spaces, that this implies that $A_{\lieq}^\circ(\mu)_{E_K(\mu)}$, as a submodule of $A_{\lieq}(\mu)_{E_K(\mu)}$, is defined over $\QQ_K(\mu)$, whence \eqref{eq:divisibilityrelation}.
\end{proof}

In the case $F=\QQ$, $s=1$, i.e.\ if there is only one even index $n_1$, and if furthermore $4\nmid n_1$, the decomposition \eqref{eq:bottomlayerdecomposition} corresponds to the decomposition of $A_{\lieq}(\mu)$ into the sum of a holomorphic and antiholomorphic discrete series representation. It also reflects the decomposition of $(\lieg,GK^0)$-cohomology that we will discuss below.

\section{Rankin-Selberg convolutions}

From this section on and for all what follows, we put ourselves in the situation where $G=G_{n+1}\times G_n$ and $H=G_n\subseteq G$ is diagonally embedded via
$$
h\;\mapsto\;\begin{pmatrix}h&\\&1\end{pmatrix}.
$$
We fix the $\QQ$-rational model of a maximal compact subgroup of $G$ as $K:=K_{n+1}\times K_n$ with $K_n$ as in section \ref{sec:compactmodels} and let $L:=H\cap K$. Then $L=K_n\subseteq G_n=H$.

Any pair of cuspidal automorphic representation $\Pi_1$ and $\Pi_2$ of $\GL_{n+1}(\Adeles_F)$ and $\GL_n(\Adeles_F)$, respectively, gives rise to an automorphic representation $\Pi$ of $G(\Adeles)$, which is the completed tensor product of $\Pi_1$ and $\Pi_2$. Call $\Pi$ {\em regular algebraic} if $\Pi_1$ and $\Pi_2$ are regular algebraic. In the same spirit we identify the local components $\Pi_v$, $v$ a place of $\QQ$, with the corresponding (completed) tensor products of the corresponding local components of $\Pi_1$ and $\Pi_2$. Representations $\Pi_\infty$, $\Pi_{1,\infty}$, $\Pi_{2,\infty}$ at $\infty$ are always assumed to be smooth, i.e.\ departing possibly from a Hilbert space representation, we implicitly pass to the subspace of smooth vectors with the corresponding Fr\'echet topology. Departing from an irreducible $\Pi$, the archimedean representations then are irreducible Casselman-Wallach representations, and $\Pi_\infty$ is a completed projective tensor product $\Pi_{1,\infty}\widehat{\otimes}\Pi_{2,\infty}.$

We write $N=N_{n+1}\times N_n\subseteq G=G_{n+1}\times G_n$ for the restriction of scalars of the group of unipotent upper triangular matrices, i.e.\ $N_m\subseteq G_m$ denotes the restriction of scalars of the group of unipotent upper triangular matrices in $\GL_m$. We fix a non-trivial continuous character
$$
\psi:\quad N(\QQ)\backslash{}N(\Adeles)\to\CC^\times
$$
with the property that its restriction to $N(\QQ)\cap H(\QQ)\backslash{}N(\Adeles)\cap H(\Adeles)$ be trivial. To be more concrete, we assume $\psi$ to be of the form
$$
\begin{pmatrix}
(u_{ij})_{1\leq i,j\leq n+1} \times (v_{ij})_{1\leq i,j\leq n}\;\mapsto\;
\prod_{k=1}^n \psi(u_{kk+1})
\cdot
\prod_{l=1}^{n-1}\psi(v_{ll+1})^{-1}
\end{pmatrix}.
$$
We fix Haar measures $dg$, $dn$, $dh$ on $G(\Adeles)$, $N(\Adeles)$ and $H(\Adeles)$, and use subscripts $(\cdot)_v$ to denote their local factors at a place $v$.

We denote by $\mathscr{W}(\Pi,\psi)$ the $\psi$-Whittaker model of $\Pi$, and we are in particular concerned with the model at infinity $\mathscr{W}(\Pi_\infty,\psi_\infty)$. Since $\Pi$ is a cuspidal representation of $G(\Adeles)$, $\Pi$ is globally generic, and hence the Whittaker model exists \cite{shalika1974}. We are interested in the Rankin-Selberg $L$-function $L(s,\Pi)$ in the sense of Jacquet, Shalika and Piatetski-Shapiro \cite{jpss1983}.

\subsection{The archimedean zeta integral}\label{sec:archimedeanintegral}

For each quasi-character $\chi\in\Hom(H(\RR),\CC^\times)$ and each $w\in\mathscr{W}(\Pi_\infty,\psi_\infty)$, the archimedean Rankin-Selberg integral
\begin{equation}
\Psi_\infty:\quad (\chi,w)\;\mapsto\;\int_{N(\RR)\backslash{}H(\RR)}w(h_\infty)\chi(h_\infty)dh_\infty
\label{eq:archimedeanintegral}
\end{equation}
converges absolutely, provided $\chi$ lies a suitable right half plane in the sense of \eqref{eq:quasicharacterchart}. In such a half plane $\Psi_\infty(-,w)$ defines a holomorphic function in $\chi$. More concretely the map
\begin{equation}
e_\infty:\quad (\chi,w)\;\mapsto\;\frac{\Psi_\infty(\chi,w)}{L(\frac{1}{2},\Pi_\infty\otimes\chi)}
\label{eq:archimedeanintegralratio}
\end{equation}
is well defined for {\em all} $\chi$, since the right hand side defines an entire function in $\chi$, cf.\ Theorem 1.2 in \cite{cogdellpiatetskishapiro2004}
(see also \cite{jacquet2009}), $\Hom(H(\RR),\CC^\times)$ being a disjoint union of finitely many copies of $\CC$. For each $\chi$ we may find a $w\in\mathscr W(\Pi_\infty,\psi_\infty)^{(K)}$ with the property that
\begin{equation}
e_\infty(\chi,w)\;\neq\; 0.
\label{eq:archimedeanintegralnonzero}
\end{equation}
Following an argument of Jacquet and Shalika (see the remark after Theorem 1.3 in \cite{cogdellpiatetskishapiro2004}), we know that for fixed $K$-finite $w$, the map
$$
\chi\;\mapsto\; e_\infty(\chi,w)
$$
is on each connected component of $\Hom(H(\RR),\CC^\times)$ given by a polynomial in the complex parameter of that component. This in particular implies that $e_\infty(-,w)$ is locally constant in the variable $\chi$ whenever \eqref{eq:archimedeanintegralnonzero} is satisfied for {\em all} quasi-characters $\chi$. In other words, for such a $w$ we know that the Rankin-Selberg integral \eqref{eq:archimedeanintegral} represents the $\Gamma$-factor of the Rankin-Selberg $L$-function $L(s,\Pi\otimes\chi)$ up to a complex unit, the latter only depending on the connected component containing $\chi.$ We will see below that we may indeed choose a vector $w$ which satisfies \eqref{eq:archimedeanintegralnonzero} for all quasi-characters $\chi.$ If $e_\infty(\chi,w)=1,$ for all $\chi$ in the same connected component, we call $w$ a {\em good test vector} for $\Pi_\infty\otimes\chi.$

In a representation theoretic sense, the integral \eqref{eq:archimedeanintegral} for $\chi$ in a suitable right half plane and \eqref{eq:archimedeanintegralratio} for general $\chi$, therefore defines, in the variable $w$, a continuous $H(\RR)$-equivariant functional
$$
e_\infty(\chi,-):\quad\mathscr W(\Pi_\infty,\psi_\infty)\to\chi^{-1}_{\CC},
$$
the right hand side denoting the one-dimensional complex $H(\RR)$-module with action given by the inverse quasi-character $\chi^{-1}.$

\begin{remark}\label{rmk:twistinvarianceatinfinity}
  Since $\Pi_\infty\otimes\chi\cong\Pi_\infty$ for every finite order character $\chi$, we have for such a $\chi$ an identity of local $L$-functions
$$
L(s,\Pi_\infty\otimes\chi)\;=\;L(s,\Pi_\infty).
$$
\end{remark}

\subsection{The non-archimedean zeta integrals}

Fix a rational prime $p$. For each quasi-character $\chi\in\Hom(H(\QQ_p),\CC^\times)$ and each $w\in\mathscr{W}(\Pi_p,\psi_p)$, the non-archimedean analogue of \eqref{eq:archimedeanintegral} is
\begin{equation}
\Psi_p:\quad (\chi,w)\;\mapsto\;\int_{N(\QQ_p)\backslash{}H(\QQ_p)}w(h_p)\chi(h_p)dh_p.
\label{eq:nonarchimedeanintegral}
\end{equation}
Again for $\chi$ in a suitable right half plane the integral in \eqref{eq:nonarchimedeanintegral} is absolutely convergent and in such a half plane $\Psi_p(-,w)$ defines a holomorphic function in $\chi$. The map
\begin{equation}
e_p:\quad (\chi,w)\;\mapsto\;\frac{\Psi_p(\chi,w)}{L(\frac{1}{2},\Pi_p\otimes\chi)}
\label{eq:nonarchimedeanintegralratio}
\end{equation}
is well defined for all $\chi$ and $w$ and defines an entire function in the variable $\chi$ on each connected component of $\Hom(H(\QQ_p),\CC^\times)$. As in the archimedean case we always find for a given $\chi$ a $w\in\mathscr W(\Pi_p,\psi_p)$ such that
$$
e_p(\chi,w)\;\neq\;0.
$$
This then implies that for each quasi-character $\chi$ there is a {\em good vector} $t_p^\chi$, depending only on the connected component of $\chi$ in $\Hom(H(\QQ_p),\CC^\times)$, satisfying
$$
e_p(\chi,t_p^\chi)\;=\;1.
$$
For almost all primes $p$ the spherical vector
$$
t_p^0\in\mathscr W(\Pi_p,\psi_p)^{G(\ZZ_p)}
$$
does the job. This is at least the case if $\Pi_p$ and $\chi_p$ are unramified, $\psi_p$ has conductor $\Oo\otimes\ZZ_p$, and
$$
\int_{H(\ZZ_p)}dh_p\;=\;1,
$$
which is always satisfied for almost all primes $p$.

\subsection{The global zeta integral}

By the preceeding discussion we may choose for each global quasi-character
$$
\chi:\quad H(\QQ)\backslash H(\Adeles)\to\CC^\times,
$$
a factorizable vector
$$
t^\chi\;=\;\otimes_{v}t_v^{\chi_v}\;\in\;\mathscr W(\Pi,\psi)^{(K)},
$$
where $v$ runs through the finite places of $\QQ$ and for almost all (finite) places $v$,
$$
t_v^{\chi_v}\;=\;t_v^0.
$$
Then the inverse Fourier transform
$$
\Theta(t^{\chi}):\quad g\;\mapsto\;
\sum_{\gamma\in G_n(\QQ)\times G_{n-1}(\QQ)}
t^\chi\left(
\begin{pmatrix}
\gamma&\\
&1
\end{pmatrix}\cdot g\right)\;\in\;\Pi^{(K)}
$$
gives rise to the global integral representation
$$
\Lambda(s,\Pi\otimes\chi)\;=\;I(s,\chi,\Theta(t^\chi)):=
\int_{H(\QQ)\backslash H(\Adeles)}\Theta(t^\chi)(h)\chi\otimes\omega_{s-\frac{1}{2}}^0(h)dh,
$$
of the completed $L$-function which converges absolutely for every $s\in\CC$ and thus defines an entire function in $s$. For $s$ in a suitable right half plane the global integral decomposes into the infinite Euler product of the local zeta integrals \eqref{eq:archimedeanintegral} and \eqref{eq:nonarchimedeanintegral}, and we write $L(s,\Pi\otimes\chi)$ for the corresponding incomplete $L$-function.

\section{Period relations}

In this section we proof the expected period relations, that we will compare to Deligne's Conjecture in the motivic context in the last section.

\subsection{Arithmeticity conditions}\label{sec:arithmeticity}

Following the terminology of \cite[Section 3]{januszewski2015}, adapted to the totally real case, we call an absolutely irreducible rational $G$-module $M$ over $E/\QQ$ {\em arithmetic}, if all its Galois twists $M^{\sigma}$ ares essentially conjugate self-dual over $\QQ,$ i.e.
$$
\left(M^{\sigma}\right)^{{\rm c},\vee}\;\cong\;M^\sigma\otimes\xi,
$$
with a $\QQ$-rational character $\xi\in X_\QQ(G)$ independent of $\sigma\in\Gal(\overline{\QQ}/\QQ).$ We remark that if $F$ is totally real, all its Galois twists $M^\sigma,$ $\sigma\in\Gal(\overline{\QQ}/\QQ)$ are automatically (conjugate) self-dual if one if its twists has this property.

The center $Z$ of $G$ may be naturally identified with $G_1\times G_1$. We may identify it with a factor of the maximal torus $T\subseteq G$, i.e.\
$$
T\;=\;
T^{\rm der}\cdot (G_1\times G_1),
$$
where $T^{\rm der}=T\cap G^{\rm der}.$ The decomposition \eqref{eq:rationaltorusdecomposition} of $G_1$ induces a natural projection
$$
p_T:\quad 
T\to T/(T^{\rm der}\cdot (G_1^{\rm an}\times G_1^{\rm an}))=:Z^{\rm s}\;\cong\;\GL_1\times\GL_1,
$$
and thus a monomorphism
$$
p_T^*:\quad X_\QQ(Z^{\rm s})\to X_\QQ(T).
$$
We fix the identification
$$
Z^{\rm s}\;=\;
\GL_1\times\GL_1,
$$
once and for all in such a way that each factor of $\GL_1$ corresponds to the corresponding factor of the center of $G$. This gives us an identification $X_\QQ(Z^{\rm s})=\ZZ^2,$ and we simply write $(w_1,w_2)$ for the image of $(w_1,w_2)\in X_\QQ(Z^{\rm s})$ under $p_T^*$.

If $M$ is of highest weight $\mu$, we denote by $\mu^{\vee}$ the highest weight of $M^{\vee}$. Then $M$ is arithmetic if and only if
$$
\left(\mu^\sigma\right)^{\rm c,\vee}-\mu^\sigma\;\in\;\image(p_T^*)
$$
and this character is independent of $\sigma\in\Gal(\overline{\QQ}/\QQ).$ Then
\begin{equation}
(\mu^\sigma)^{\rm c,\vee}-\mu^\sigma
\;=\;
(w_1,w_2),
\label{eq:prearithmeticitiyweights}
\end{equation}
with $(w_1,w_2)$ independent of $\sigma\in\Gal(\overline{\QQ}/\QQ).$

A $\QQ$-rational character $\xi\in X_\QQ(H)\cong\ZZ$ is called {\em critical} for $M$ if
$$
\Hom_H(M,\xi)\;\neq\;0.
$$
Due to the multiplicity one property of the pair $\GL_{n+1}|\GL_n$ (cf.\ the non-compact analogue of Proposition \ref{prop:bonedimensionalfunctionals} below), we have
\begin{equation}
\dim_{\QQ_K(\mu)}\Hom_H(M_{\QQ_K(\mu)},\xi_{\QQ_K(\mu)})\;=\;1,
\label{eq:criticaldimensionone}
\end{equation}
for each critical $\xi$. We call $M$ {\em balanced} if it admits a critical character $\xi$.

\subsection{Rational test vectors}

We specialize the notation of section \ref{sec:cohomologicalinduction} to the case
\begin{eqnarray*}
n_1&=&
\begin{cases}
n,&2\mid n,\\
n+1,&2\nmid n,
\end{cases}\\
n_2&=&\frac{(n+1)n}{n_1},
\end{eqnarray*}
and $r=2$. Hence, $n_1$ is even and $n_2$ is odd. Recall $E_K=\QQ_K(\sqrt{-1})$ as before and fix a $\theta$-stable Borel $\lieq\subseteq\lieg$, which is a product of two $\theta$-stable parabolic subalgebras in each factor of $G=G_{n+1}\times G_n$.
We assume $\lieq$ to be transversal to the Lie algebra $\lieh$ of $H$, i.e.\
\begin{equation}
\lieg_{E_K}\;=\;\lieq_{E_K}\oplus \lieh_{E_K}.
\label{eq:transversality}
\end{equation}

We assume that $\Pi_\infty$ has non-trivial relative Lie algebra cohomology with coefficients in $M(\mu)^\vee$. We know that $M(\mu)$ is arithmetic if such a cuspidal $\Pi$ exists. By \cite{voganzuckerman1984} that we have an isomorphism
\begin{equation}
\iota_\infty:\quad A_\lieq(\mu)_\CC\;\to\;\mathscr W(\Pi_\infty,\Psi_\infty)^{(K)}
\label{eq:moduleisomorphism}
\end{equation}
where the base change on the left hand side is implicitly understood via the fixed embedding \eqref{eq:complexembedding}.

For any real archimedean place $v$ of $F,$ fix an $l_v\in K_n(\QQ_K)$ with the property that its image in
$$
\pi_0(K_n)\;\cong\;\{\pm1\}^{r_F^\RR}
$$
is non-trivial in the factor $\{\pm1\}$ corresponding to $v$ and trivial in the factors corresponding to other real places. Assume without loss of generality that $l_v$ is chosen in the image of a homomorphic section
$$
\pi_0(K_n)\to K_n.
$$
We consider $l_v$ as an element of $L(\QQ)$ via the respective identification
\begin{equation}
G_n\;=\;L.
\label{eq:gnembedding}
\end{equation}
Then the image of $l_v$ in $\pi_0(K)$ corresponds likewise under the isomorphism
$$
\pi_0(K)\;\cong\;
\{\pm1\}^{r_F^\RR}\times\{\pm1\}^{r_F^\RR},
$$
to the element whose $v$-component is $(-1,-1)$ and trivial otherwise. Therefore, the action of $l_v$ on $B_{\lieq}(\mu)$ interchanges the two modules $B_{\mu_{v,1}}^{\SO}\otimes B_{\mu_{v,2}}^{\SO}$ and $B_{\tilde{\mu}_{v,1}}^{\SO}\otimes B_{\mu_{v,2}}^{\SO}$ and leaves the other factors $B_{\mu_{v',1}}^{\bullet}\otimes B_{\mu_{v',2}}^{\bullet},$ $v'\neq v,$ in \eqref{eq:bottomlayerzerodecomposition} invariant. In particular, this shows that any irreducible $K^0$-submodule of $B_{\lieq}(\mu)_\CC$ generates $B_{\lieq}(\mu)_\CC$ as a representation of $L$ and even of $L\cap H_1$.

The transversality condition \eqref{eq:transversality} yields
$$
\liek_{\CC}\;=\;\liel_{\CC}\oplus(\lieq_{\CC}\cap\liek_{\CC})
$$
and by the Poincar\'e-Birkhoff-Witt Theorem we have a canonical isomorphism
\begin{equation}
U(\liek_{\CC})\;=\;U(\liel_{\CC})\otimes U(\lieq_{\CC}\cap\liek_{\CC})
\label{eq:utransversality}
\end{equation}
of $\CC$-vector spaces. As a consequence, we obtain

\begin{proposition}\label{prop:bonedimensionalfunctionals}
For each character $\chi\in X_\CC(L),$
$$
\dim_\CC\Hom_L(B_{\lieq}(\mu)_\CC,\chi_\CC)\;\leq\;1.
$$
\end{proposition}

\begin{proof}
We sketch a proof for the convenience of the reader, which is an adaption of an argument in \cite[Lemma 2.10]{sunpreprint2}. We show first the analoguous statement
$$
\dim_\CC\Hom_{L^0}(B_{\lieq}^\circ(\mu)_\CC,\chi_\CC)\;\leq\;1.
$$
To this point, let
$$
\lambda\;\in\;\Hom_{L^0}(B_{\lieq}^\circ(\mu)_\CC,\chi_\CC).
$$
Assume that the restriction of $\lambda$ to the one-dimensional highest weight space
$$
H^0(\lieu_\CC\cap\liek_\CC; B_{\lieq}^\circ(\mu)_\CC)
$$
vanishes. Since any non-zero vector of this space generates $B^{\circ}(\mu)_\CC$ as a $U(\liek_\CC)$-module, we see with \eqref{eq:utransversality} that $\lambda$ must vanish. Hence the space of $L^0$-equivariant $\lambda:B_{\lieq}^\circ(\mu)\to\chi$ is at most one-dimensional.

Since $B_\lieq^\circ(\mu)_\CC$ generates $B_\lieq(\mu)_\CC$ as an $L$-module, the claim follows.
\end{proof}

\begin{proposition}\label{prop:finiteness}
For $n\geq 1$ there exists a vector
$$
t_0\;\in\;B_{\lieq}(\mu)_{\QQ_K(\mu)}
$$
with the property that for any character $\chi\in X_{\CC}(L)$ and any
$$
0\neq\lambda\;\in\;\Hom_{L}(B_{\lieq}(\mu)_\CC,\chi_\CC)
$$
we have
$$
\lambda(t_0)\;\neq\; 0.
$$
\end{proposition}

\begin{proof}
First remark that, since $B_\lieq(\mu)_\CC$ is completely reducible as an $L$-module, there are only finitely many characters $\chi\in X_\CC(L)$ with the property
\begin{equation}
\Hom_L(B_\lieq(\mu)_\CC,\chi_\CC)\;\neq\;0.
\label{eq:chioccursinb}
\end{equation}
By Proposition \ref{prop:bonedimensionalfunctionals}, the union of the kernels of all non-zero functionals $\lambda$ for the finitely many $\chi\in X_{\CC}(L)$ satisfying \eqref{eq:chioccursinb} is a Zariski closed subset of $B_\lieq(\mu)_\CC$ of codimension $1$. In particular, its complement $U\subseteq B_\lieq(\mu)_\CC$ is non-empty and open for the standard topology on $B_\lieq(\mu)_\CC$ as a finite-dimensional topological $\CC$-vector space. Since $B_{\lieq}(\mu)_{\QQ_K(\mu)}$ is dense in $B_{\lieq}(\mu)_{\CC}$, we find an element
$$
t_0\;\in\;U\cap B_{\lieq}(\mu)_{\QQ_K(\mu)},
$$
as desired. This concludes the proof.
\end{proof}

Each one-dimensional $(\lieh,L)_\CC$-module $\chi_\CC$ corresponds bijectively to a quasi-character $\chi$ of $H(\RR)$ and by composing the functional $e_\infty(\chi,\cdot)$ of section \ref{sec:archimedeanintegral} with the fixed isomorphism $\iota_\infty$ in \eqref{eq:moduleisomorphism} we obtain a non-zero $(\lieh,L)_\CC$-equivariant functional
\begin{equation}
\lambda_{\chi,\CC}:= e_\infty(\chi,\cdot)\circ\iota_\infty:\quad A_\lieq(\mu)_\CC\to \chi^{-1}_{\CC}.
\label{eq:complexfunctional}
\end{equation}
We formulate the following
\begin{conjecture}\label{conj:rankinfunctional0}
For each $n\geq 1$ and each quasi-character $\chi$ the functional $\lambda_{\chi,\CC}$ is non-zero on the minimal $K$-type $B_{\lieq}(\mu)_\CC.$
\end{conjecture}

\begin{proposition}\label{prop:rationaltestvector}.
  Assume the validity of Conjecture \ref{conj:rankinfunctional0} for $F,$ $n,$ and $\mu.$
  Then there exists a rational test vector $t_0\in B_\lieq(\mu)_{\QQ_K(\mu)}$ with the property that for every $\chi\in\Hom(H(\RR),\CC^\times)$,
\begin{equation}
\lambda_{\chi,\CC}(t_0)\;\neq\; 0.
\label{eq:testvectorcondition}
\end{equation}
In particular, for every quasi-character $\chi$ there is a constant $c_\chi\in\CC^\times$, only depending on the connected component containing $\chi$ in $\Hom(H(\RR),\CC^\times)$, with the property that
$$
e_\infty(\chi,\iota_\infty(t_0))\;=\;c_\chi.
$$
\end{proposition}

\begin{proof}
  By the assumption,
  \begin{equation}
    0\neq\lambda_{\chi,\CC}|_{B_{\lieq}(\mu)_\CC}\in\Hom_{L}(B_{\lieq}(\mu)_\CC,\chi|_{L.\CC}^{-1}).
    \label{eq:resnonzero}
  \end{equation}
  Therefore, any choice of $t_0$ as in Proposition \ref{prop:finiteness} satisfies \eqref{eq:testvectorcondition}. By the archimedean Rankin-Selberg theory that we discussed in section \ref{sec:archimedeanintegral}, the condition
  $$
    \forall\chi:\quad e_\infty(\chi,\iota_\infty(t_0))\;\neq\;0
  $$
  is satisfied for all $\chi$ and this implies the claim.
\end{proof}

\begin{corollary}\label{cor:minimalktypetestvector}
  Assume the vailidity of Conjecture \ref{conj:rankinfunctional0} for $F,$ $n,$ and $\mu.$
  For any $t\in B_{\lieq}(\mu)_{\QQ_K(\mu)}$ and any algebraic quasi-character $\chi\in X^{\rm alg}(H(\RR))$ we have
$$
e_\infty(\chi,\iota_\infty(t))\;\in\;\QQ_K(\mu)\cdot c_\chi.
$$
\end{corollary}

\begin{proof}
  Since $\chi$ is defined over $\QQ$, we see with Proposition \ref{prop:bonedimensionalfunctionals} that $\lambda_{\CC,\chi}$ restricted to $B_{\lieq}(\mu)_{\QQ_K(\mu)}$ is defined over $\QQ_K(\mu).$
  Therefore, the image of
  $$
    B_{\lieq}(\mu)_{\QQ_K(\mu)}\;\subseteq\; A_{\lieq}(\mu)_{\CC}
  $$
  under $\lambda_{\chi,\CC}$ is a one-dimensional $\QQ_K(\mu)$-subspace of $\chi_\CC^{-1}\cong\CC$. This subspace contains $\lambda_{\chi,\CC}(t)$ and $\lambda_{\chi,\CC}(t_0)$. Therefore the claim follows from Proposition \ref{prop:rationaltestvector}.
\end{proof}

We also formulate
\begin{conjecture}\label{conj:rankinfunctional}
  For any number field $F/\QQ,$ any $n\geq 1$ and any balanced $\mu,$ there is a rational test vector $t\in A_{\lieq}(\mu)_{\QQ_K(\mu)}$, which is good up to a complex unit, and for each critical quasi-character $\chi$, $\lambda_{\chi,\CC}$ is defined over $\QQ_K(\mu)$.
\end{conjecture}

\begin{remark}
  Under Conjecture \ref{conj:rankinfunctional}, the conclusion of Corollary \ref{cor:minimalktypetestvector} holds for all $t\in A_{\lieq}(\mu)_{\QQ_K(\mu)}.$
\end{remark}

\begin{theorem}\label{thm:conjectures}
  Conjectures \ref{conj:rankinfunctional0} and \ref{conj:rankinfunctional} are true in the following cases: $n=1$ and general number fields $F,$ and $1\leq n\leq 2$ and $F$ totally real.
\end{theorem}

\begin{proof}
  We treat the case $F$ totally real and $n=2$ first.

  In representation theoretic terms Theorem 4.4 in \cite{kastenschmidt2008} implies that for any quasi-character $\chi$ of $H(\RR),$ the restriction map
  \begin{equation}
    \Hom_{\lieh,L}(A_\lieq(\mu)_\CC,\chi_\CC)\to \Hom_{L}(B_{\lieq}(\mu)_\CC,\chi|_{L,\CC})
    \label{eq:functionalres}
  \end{equation}
  is a monomorphism. This proves Conjecture \ref{conj:rankinfunctional0}. By Proposition \ref{prop:rationaltestvector}, this also shows the existence of a $\QQ_K(\mu)$-rational good test vector.

  By Proposition \ref{prop:bonedimensionalfunctionals}, the right hand side of \eqref{eq:functionalres} is at most one-dimensional, i.e.\ we obtain
  \begin{equation}
    \dim_\CC\Hom_{\lieh,L}(A_{\lieq}(\mu)_\CC,\chi_{\CC})\;\leq\;1.
    \label{eq:functionaldimension}
  \end{equation}
  (The non-vanishing of archimedean Rankin-Selberg zeta integrals in \eqref{eq:archimedeanintegralnonzero} implies that this is dimension is always one.)

  Therefore, we conclude from \eqref{eq:functionaldimension} with Proposition 1.1 in \cite{januszewski2017} that $\lambda_{\chi,\CC}$ is defined over $\QQ_K(\mu).$

  The case $n=1$ may be treated similarlyr by specializing the general arguments given in \cite{januszewskipart2} for $\GL(2n)$ (which extends to $\GL(2)$ over number fields). Conjecture \ref{conj:rankinfunctional0} also follows from the discussions in \cite{hida1994} and \cite{namikawa2016}.
\end{proof}

\begin{remark}
  The statement of Conjecture \ref{conj:rankinfunctional0} for arbitrary $F$ and $n$ reduces to the case $F=\QQ,$ and an instance where $F/\QQ$ is imaginary quadratic.

  The statement in Conjecture \ref{conj:rankinfunctional} is may be considered weaker than that of Conjecture \ref{conj:rankinfunctional0}, although from a rationality perspective the statement is stronger than Conjecture \ref{conj:rankinfunctional0}.
  \end{remark}

\ \\
{\bf For $n\geq 3$ or $n\geq2$ whenever $F$ admits a complex place, our results in the rest of the paper depend conditionally on Conjecture \ref{conj:rankinfunctional0} or Conjecture \ref{conj:rankinfunctional}.}\\\ 

The unconditional case $n=2$ corresponds to $G=\res_{F/\QQ}(\GL_3\times\GL_2)$ for $F/\QQ$ totally real, in which case our results are new. The case $n=1$ allows for arbitrary $F$ and corresponds to $G=\res_{F/\QQ}\GL(2),$ where we provide another argument for the results Manin, Shimura and Hida \cite{manin1972,manin1976,shimura1976,shimura1977,shimura1978,hida1994}.

\subsection{The archimedean period relation}

We henceforth assume Conjecture \ref{conj:rankinfunctional0} or Conjecture \ref{conj:rankinfunctional}. Our first result towards period relations is

\begin{theorem}\label{thm:icontribution}
For any algebraic quasi-character $\chi$ of $H(\RR)$ we have
$$
c_{\sgn_\infty\otimes\chi}\;\in\;\QQ_K(\mu)\cdot (i^{r_F^\RR m}\cdot c_\chi),
$$
where
$$
m\;:=\;
\frac{(n+1)n}{2}.
$$
\end{theorem}

\begin{proof}
  Assume $r_F^\RR>0,$ otherwise the statement is clear. According to the direct sum decomposition \eqref{eq:standardmoduledecomposition}, or equivalently \eqref{eq:bottomlayerdecomposition}, which is already defined over $E_K(\mu)$, we may write
\begin{equation}
t\;:=\;t_0\;=\;
\sum_{\varepsilon\in\pi_0(H)}t_\varepsilon
\label{eq:tsumrepresentation}
\end{equation}
where
$$
t_\varepsilon\;\in\;\varepsilon\cdot B_\lieq^\circ(\mu)_{E_K(\mu)}\;\subseteq\;B_\lieq(\mu)_{E_K(\mu)},
$$
for our choice of representatives
$$
\varepsilon\;=\;\prod_{v}l_v^{\delta_v}\;\in\; L(\QQ),\quad \delta_v\in\{0,1\}.
$$
Since these elements form a system of representatives of
$$
\pi_0(H(\RR))\;=\;\pi_0(G(\RR)/C(\RR)),
$$
we have
$$
G(\RR)\;=\;\bigsqcup_{\varepsilon} G(\RR)^0C(\RR)\cdot\varepsilon.
$$
Considering $A_{\lieq}(\mu)_\CC$ as a $(\lieg,K^0)_\CC$-module, we see that $\iota_\infty(t_\varepsilon)$, as a function on $G(\RR)$, has support in a set of the form $G(\RR)^0C(\RR)\cdot\varepsilon'$, and 
$$
\varepsilon'\;=\;\varepsilon\cdot\varepsilon_0,
$$
with a representative $\varepsilon_0$ independent of $\varepsilon$. We remark that $\varepsilon_0=1$ thanks to \eqref{eq:aqintermediateinduction}, which is compatible with parabolic induction.

Accordingly, the archimedean Rankin-Selberg integral \eqref{eq:archimedeanintegral} decomposes into the sum of the integrals
$$
\Psi_\infty^\varepsilon:
\quad (\chi,w)\;\mapsto\;
\int_{N(\RR)\backslash{}H(\RR)^0\varepsilon}w(h_\infty)\chi(h_\infty)dh_\infty,
$$
over the individual connected compontents.

These integrals are convergent for $\chi$ in a suitable right half plane. Similarly, we have the ratios
$$
e_\infty^\varepsilon:\quad (\chi,w)\;\mapsto\;\frac{\Psi_\infty^\varepsilon(\chi,w)}{L(\frac{1}{2},\Pi_\infty\otimes\chi)},
$$
which are again entire functions in $\chi$, and
$$
e_\infty(\chi,\iota_\infty(t))\;=\;\sum_{\varepsilon} e_\infty^{\varepsilon}(\chi,\iota_\infty(t_{\varepsilon})).
$$
Since $dh_\infty$ is a Haar measure, we get the relation
\begin{equation}
e_\infty^\varepsilon(\chi,\iota_\infty(t))\;=\;
\chi(\varepsilon)\cdot e_\infty^{\bf1}(\chi|_{H(\RR)^0},\varepsilon\cdot \iota_\infty(t_{\varepsilon})).
\label{eq:echipartial}
\end{equation}
With relation \eqref{eq:echipartial} we conclude that for each real place $v$ of $F$,
\begin{eqnarray*}
\lambda_{\chi,E_K(\mu)}(t_\varepsilon)
&=&
e_\infty^{\varepsilon}(\chi,\iota_\infty(t))\\
&=&\chi(\varepsilon)\cdot e_\infty^{\bf1}(\chi|_{H(\RR)^0},\varepsilon\cdot \iota_\infty(t_\varepsilon))\\
&=&\chi(\varepsilon)\cdot e_\infty^{\bf1}((\chi\otimes\sgn_v)|_{H(\RR)^0},\varepsilon\cdot \iota_\infty(t_\varepsilon))\\
&=&\sgn_v(\varepsilon)\cdot
 e_\infty^{\varepsilon}(\chi\otimes\sgn_v,\iota_\infty(t_\varepsilon))\\
&=&\sgn_v(\varepsilon)\cdot\lambda_{\chi\otimes\sgn_v,E_K(\mu)}(t_\varepsilon).
\end{eqnarray*}
Summing up, we obtain
\begin{equation}
\lambda_{\chi,E_K(\mu)}(t)\;=\;
\sum_{\varepsilon}\lambda_{\chi,E_K(\mu)}(t_\varepsilon)\;=\;
\sum_{\varepsilon}\sgn_v(\varepsilon)\cdot\lambda_{\chi\otimes\sgn_v,E_K(\mu)}(t_\varepsilon).
\label{eq:integralcomparison}
\end{equation}

Complex conjugation $\tau_K\in\Gal(E_K/\QQ)$ leaves the direct sum decompositions \eqref{eq:standardmoduledecomposition} and \eqref{eq:bottomlayerdecomposition} invariant, but permutes the direct factors. By Theorem \ref{thm:connectedbottomlayerdefinition}, this action is trivial if and only if $2\mid m$, i.e.\ if and only if
$$
i^{m}\;\in\;\QQ.
$$
Let us suppose $\sqrt{-1}\not\in\QQ_K(\mu)$ and $2\nmid m$. Choose a real place $v_0$ of $F$ and consider the $K^0$-submodule
$$
B_{v_0,E_K(\mu)}\;:=\;\sum_{\varepsilon\in\kernel\sgn_{v_0}}\varepsilon\cdot B_{\lieq}^\circ(\mu)_{E_K(\mu)}\;\subseteq\;B_{\lieq}(\mu)_{E_K(\mu)},
$$
where the sum ranges over all possible products $\varepsilon$ of the elements $l_v$ with $v\neq v_0$. Then
\begin{equation}
B_{\lieq}(\mu)_{E_K(\mu)}\;=\;B_{v_0,E_K(\mu)}\oplus l_{v_0}\cdot B_{v_0,E_K(\mu)}.
\label{eq:bottomlayervdecomposition}
\end{equation}
The second direct summand on the right hand side is naturally identified with the dual of $B_{v_0,E_K(\mu)}$ due to our hypothesis $2\nmid m$. We conclude that $\tau_K$, as an automorphism of $E_K(\mu)/\QQ_K(\mu)$, and thus of $B_\lieq(\mu)_{E_K(\mu)}$, interchanges the two direct summands in \eqref{eq:bottomlayervdecomposition}.

Hence $\tau_K$ sends the vector
$$
t_{v_0}\;:=\;\sum_{\varepsilon\in\kernel\sgn_{v_0}} t_\varepsilon
$$
to
$$
t_{v_0}^{\tau_K}\;\in\;l_{v_0}\cdot B_{v_0,E_K(\mu)}.
$$
Now $t$ is $\QQ_K(\mu)$-rational, and thus invariant under $\tau_K$. The sum decomposition \eqref{eq:tsumrepresentation} being unique, we conclude that for each representative $\varepsilon$,
$$
t_{v_0}^{\tau_K}\;=\;\sum_{\varepsilon\in\kernel\sgn_{v_0}} t_{l_{v_0}\varepsilon}\;=:\;t_{-v_0}.
$$
Hence the vector
$$
t_{v_0}-t_{-v_0}\;\in\;i\cdot B_\lieq(\mu)_{\QQ_K(\mu)}\;\subseteq\;B_\lieq(\mu)_{E_K(\mu)}.
$$
is \lq{}purely imaginary\rq{} (in the Galois theoretic sense in the context of the non-trivial extension $E_K(\mu)/\QQ_K(\mu)$), and consequently
\begin{equation}
i\cdot(t_{v_0}-t_{-v_0})\;\in\; B_\lieq(\mu)_{\QQ_K(\mu)}\;\subseteq\;B_\lieq(\mu)_{E_K(\mu)}.
\label{eq:tauactionoontepsilon}
\end{equation}
Turning our attention to the functionals in \eqref{eq:integralcomparison}, we observe
\begin{eqnarray*}
\lambda_{\chi,E_K(\mu)}(t)
&=&
\lambda_{\chi\otimes\sgn_v,E_K(\mu)}(t_{v_0})\;-\;
\lambda_{\chi\otimes\sgn_v,E_K(\mu)}(t_{-v_0})\\
&=&
\lambda_{\chi\otimes\sgn_v,E_K(\mu)}(t_{v_0}-t_{-v_0})\\
&=&
i\cdot
\lambda_{\chi\otimes\sgn_v,E_K(\mu)}(i\cdot(t_{-v_0}-t_{v_0}))\\
&\in&
\lambda_{\chi\otimes\sgn_v,\QQ_K(\mu)}(t)\cdot i\cdot\QQ_K(\mu),
\end{eqnarray*}
where the last relation follows from \ref{eq:tauactionoontepsilon} and the rationality property of the functional. Since
$$
\sgn_\infty\;=\;\otimes_{v_0}\sgn_{v_0},
$$
iteration over the $r_F$ real places of $F$ proves the claim in the case $\sqrt{-1}\not\in\QQ_K(\mu)$ and $2\nmid m$.

If $\sqrt{-1}\in\QQ_K(\mu)$ or $2\mid m$, then the vectors $t$ and $t_{\pm v_0}$ all lie in the same $\QQ_K(\mu)$-rational model $B(\mu)_{\QQ_K(\mu)}$, and thus the claim follows in this case by the rationality of the functional as well.
\end{proof}

The proof of Theorem \ref{thm:icontribution} may be interpreted as an automorphic reflection of the motivic Corollaire 1.6 in \cite{deligne1979}. We will discuss this relation in more detail in section \ref{sec:deligne}.

\begin{corollary}\label{cor:icontribution}
For any $t\in A_{\lieq}(\mu)_{\QQ_K(\mu)},$ and any algebraic quasi-character $\chi$ of $H(\RR)$ and any $k\in\ZZ,$
$$
\lambda_{\chi[k],\QQ_K(\mu)}(t)\;\in\;\QQ_K(\mu)\cdot (i^{kr_F^\RR m}\cdot c_\chi).
$$
\end{corollary}

\begin{proof}
If $r_F^\RR>0$, we see that by \eqref{eq:algebraicarchimedeannorm}, the two quasi-characters
$$
\chi[k]\;=\;\chi\otimes\left(\mathcal N^{\otimes k}\right)
$$
and $\chi$ lie in the same connected component if and only if $2\mid k$. In the case $2\nmid k$ the character $\chi[k]$ lies in the same component as $\chi\otimes\sgn_\infty$. The corollary follows from Theorem \ref{thm:icontribution} and the constancy of $c_\chi$ and $c_{\chi\otimes\sgn_\infty}$ on connected components.
\end{proof}

We say that an algebraic $\chi=\sgn^{\delta}\otimes\left(\mathcal N^{\otimes k}\right)$ is {\em critical} for $\Pi$ (or $\Pi_\infty$), if $L(s,\Pi_\infty)$ and $L(1-s,\Pi_\infty^\vee)$ both have no pole at $s=k+\frac{1}{2}$. For critical $\chi$ we know that
$$
L_\infty(\frac{1}{2},\Pi_\infty\otimes\chi)\;\neq\;0.
$$
By Corollary \ref{cor:minimalktypetestvector} under Conjecture \ref{conj:rankinfunctional0} or by Conjecture \ref{conj:rankinfunctional}, we therefore find a $t_0\in A_{\lieq}(\mu)_{\QQ_K(\mu)}$ satisfying
$$
\Psi_\infty(\chi,\iota_\infty(t_0))\;\neq\;0,
$$
for any critical $\chi$, $\Psi_\infty(\chi,\iota_\infty(t_0))$ being defined by holomorphic continuation outside the region of absolute convergence. This implies
\begin{corollary}\label{cor:criticalrelation}
For any $t\in A_{\lieq}(\mu)_{\QQ_K(\mu)}$, any pair of critical quasi-characters $\chi,\chi'$ of $H(\RR)$ with
$$
\chi'\;=\;\chi[k]\otimes(\sgn_\infty)^{\delta},\quad k\in\ZZ,\,\delta\in\{0,1\},
$$
we have
$$
\Psi_\infty(\chi',\iota_\infty(t))\;\in\;
\QQ_K(\mu)\cdot i^{(k+\delta)r_F^\RR m}\cdot
\Psi_\infty(\chi,\iota_\infty(t_0))\cdot
\frac{L_\infty(\frac{1}{2}+k,\Pi_\infty\otimes\chi)}{L_\infty(\frac{1}{2},\Pi_\infty\otimes\chi)},
$$
for every $t\in A_{\lieq}(\mu)_{\QQ_K(\mu)}$.
\end{corollary}

\subsection{Cohomology and Galois actions}

For any $\sigma\in\Aut(\CC/\QQ_K)$ we have a twisted representation
$$
M(\mu)_{E^\sigma}^\sigma:=M(\mu)_E\otimes_{E,\sigma^{-1}}\sigma^{-1}(E)
$$
of $G$. This twisting operation is compatible with the Galois action \eqref{eq:galoismu} on highest weights. Consider the counit map
$$
\epsilon:\quad
M(\mu)\to(\res_{\QQ_K(\mu)/\QQ_K}M(\mu))\otimes_\QQ \QQ_K(\mu),
$$
where the right hand side is defined over $\QQ_K$. This map extends to an $E$-linear map
$$
\epsilon_E:\quad
M(\mu)_E\to
(\res_{\QQ_K(\mu)/\QQ_K}M(\mu))\otimes_{\QQ_K} E,
$$
which fits into a commutative diagram
$$
\begin{CD}
M(\mu)_E@>\epsilon_E>> (\res_{\QQ_K(\mu)/\QQ_K}M(\mu))_E\\
@V\sigma^{-1} VV @VV\sigma^{-1} V\\
M(\mu)_{E^\sigma}^\sigma @>\epsilon_E>> (\res_{\QQ_K(\mu)/\QQ_K}M(\mu))_{E^\sigma}\\
\end{CD}
$$
Simply put, the Galois action on $M(\mu)$ is compatible with the intrinsic Galois action on the $\CC$-valued points of the $\QQ_K$-rational module $\res_{\QQ_K(\mu)/\QQ_K}M(\mu)$. In particular, the latter may be thought of as the sum of the Galois conjugates of $M(\mu)$.

Now the same discussion applies mutatis mutandis to $A_{\lieq}(\mu)$ and $A_{\lieq}(\mu)\otimes M(\mu)^\vee$ instead of $M(\mu)$. Note that restricton of scalars does not commute with tensor products, but we have natural isomorphisms of $(\lieg,K)$-modules over $\QQ_K(\mu)$:
\begin{eqnarray*}
\res_{\QQ_K(\mu)/\QQ}(A_{\lieq}(\mu)\otimes M(\mu)^\vee)_{\QQ_K(\mu)}&=&
\res_{\QQ_K(\mu)/\QQ}(A_{\lieq}(\mu))_{\QQ_K(\mu)}\otimes M(\mu)^\vee\\
&=& A_{\lieq}(\mu)\otimes\res_{\QQ_K(\mu)/\QQ}(M(\mu)^\vee)_{\QQ_K(\mu)}.
\end{eqnarray*}
We introduce the $\QQ_K$-group
$$
GK=\{g\in G\mid \exists z\in Z^{\rm s}:\;zg=\theta(g)\}\;\subseteq\;G.
$$
It is the product of $K$ with the maximal $\QQ$-split torus in the center $Z$ of $G$. The natural isomorphism
$$
K/K^0C\;=\;GK/GK^0C\;=\;\pi_0(L),
$$
in light of \eqref{eq:aqintermediateinduction}, together with Shapiro's Lemma, implies
\begin{eqnarray*}
H^{\bullet}(\lieg,GK^0; A_{\lieq}(\mu)\otimes M(\mu)^\vee)_{E_K(\mu)}&=&
H^{\bullet}(\lieg,GK^0; \Gamma_{\lieg,CK^0}^{\lieg,K}(A_{\lieq}^\circ(\mu)\otimes M(\mu)^\vee))_{E_K(\mu)}\\
&=&
\ind_{\pi_0(L^0)}^{\pi_0(L)}H^{\bullet}(\lieg,GK^0; A_{\lieq}^\circ(\mu)\otimes M(\mu)^\vee)_{E_K(\mu)}\\
&=&
H^{\bullet}(\lieg,GK^0; A_{\lieq}^\circ(\mu)\otimes M(\mu)^\vee)_{E_K(\mu)}\otimes \QQ[\pi_0(L)]\\
&=:&
H^{\bullet}(\lieg,GK^0; A_{\lieq}^\circ(\mu)\otimes M(\mu)^\vee)[\pi_0(L)]_{E_K(\mu)},
\end{eqnarray*}
as $\pi_0(L)$-modules.

Introduce the bottom degrees
$$
b_n^{\RR}\;:=\;
\left\lfloor
\frac{n^2}{4}
\right\rfloor,\quad
b_n^{\CC}\;:=\;\frac{n(n-1)}{2},
$$
which are the lowest degrees for which the relative Lie algebra cohomology of non-degenerate cohomological representations of $\GL_n(\RR)$ and $\GL_n(\CC)$ does not vanish, and set
$$
d\;:=\;
\sum_{v}b_{n+1}^{F_v}+b_n^{F_v}.
$$
This is the bottom degree of Lie algebra cohomology for non-degenerate cohomological representations of $G(\RR)$. Since the cohomology of $A_{\lieq}^\circ(\mu)$ in the degree $d$ is one-dimensional, the standard descent argument \cite[Proposition 1.1]{januszewski2017} together with the Homological Base Change Theorem in loc.\ cit.\ shows that we have a natural isomorphism of $\pi_0(L)$-modules
\begin{equation}
H^{d}(\lieg,GK^0; A_{\lieq}(\mu)\otimes M(\mu)^\vee)_{\QQ_K(\mu)}\;=\;
\QQ_K(\mu)[\pi_0(L)].
\label{eq:cohomologycomponentaction}
\end{equation}
Applying the same restriction of scalars argument again, we see that 
$$
H^{d}(\lieg,GK^0; \res_{\QQ_K(\mu)/\QQ}(A_{\lieq}(\mu)\otimes M(\mu)^\vee))\;=\;
\res_{\QQ_K(\mu)/\QQ}\QQ_K(\mu)[\pi_0(L)]
$$
is defined over $\QQ_K$. Over $\CC$ it comes with a natural action of $\Aut(\CC/\QQ_K)$, and this action has an extension to a global $\QQ_K$-structure on the space of regular algebraic cusp forms, cf.\ Theorem A of \cite{januszewski2017}.

\subsection{Cohomological test vectors}\label{sec:cohomologicaltestvectors}

We have a natural isomorphism
\begin{equation}
H^{\bullet}(\lieg,GK^0; A_{\lieq}(\mu)\otimes M(\mu)^\vee)_{\QQ_K(\mu)}\;=\;
H^0(GK^0; \bigwedge^\bullet(\lieg/\liegk)^\vee\otimes 
A_{\lieq}(\mu)\otimes M(\mu)^\vee)_{\QQ_K(\mu)}
\label{eq:cohomologyrepresentation}
\end{equation}
of $\pi_0(L)$-modules, since the standard complex computing the relative Lie algebra cohomology degenerates in our case. This is well known over $\CC$ (cf.\ combine Proposition 3.1 in Borel-Wallach \cite{borelwallach2000} with Proposition 3.1 in \cite{januszewski2017}), and this already implies the claim over $\QQ_K(\mu)$.

The structure of $(\lieg,K)$-cohomology has been studied in general by Vogan and Zuckerman in \cite{voganzuckerman1984}. In particular, the canonical embedding
$$
H^0(GK^0; \bigwedge^\bullet(\lieg/\liegk)^\vee\otimes 
B_{\lieq}(\mu)\otimes M(\mu)^\vee)_{\QQ_K(\mu)}\;\to\;
H^0(GK^0; \bigwedge^\bullet(\lieg/\liegk)^\vee\otimes 
A_{\lieq}(\mu)\otimes M(\mu)^\vee)_{\QQ_K(\mu)}
$$
is an isomorphism. Any cohomology class
$$
h\;\in\;H^{d}(\lieg,GK^0; A_{\lieq}(\mu)\otimes M(\mu)^\vee)_{\QQ_K(\mu)}
$$
has by \eqref{eq:cohomologyrepresentation} a unique representative
$$
h\;=\;
\sum_{p=1}^s \omega_p\otimes a_p\otimes m_p
\;\in\;H^0(GK^0; \bigwedge^d(\lieg/\liegk)^\vee\otimes 
B_{\lieq}(\mu)\otimes M(\mu)^\vee)_{\QQ_K(\mu)},
$$
with
$$
\omega_p\;\in\;\bigwedge^d(\lieg/\liegk)_{\QQ_K(\mu)}^\vee,\quad
a_p\;\in\;B_{\lieq}(\mu)_{\QQ_K(\mu)},\quad
m_p\;\in\;M(\mu)_{\QQ_K(\mu)}^\vee,\quad 1\leq p\leq s.
$$
Identifying $Z^{\rm s}$ with the maximal $\QQ$-split torus in the center of $G$ and observe that $H\cap Z^{\rm s}=1.$

Consider the diagonal embedding $(\lieh/\liel)_{\QQ_K(\mu)}\to(\lieg/\liegk)_{\QQ_K(\mu)},$ which dually induces a projection $(\lieg/\liegk)_{\QQ_K(\mu)}^\vee\to(\lieh/\liel)_{\QQ_K(\mu)}^\vee.$ The $d$-th exterior power of this map gives rise to the restriction map
$$
\res^G_H:\quad \bigwedge^d(\lieg/\liegk)_{\QQ_K(\mu)}^\vee\;\to\;
\bigwedge^d(\lieh/\liel)_{\QQ_K(\mu)}^\vee,
$$
where the right hand side is one-dimensional due to the numerical coincidence $d=\dim\lieh/\liel.$ We fix a $\QQ_K(\mu)$-rational basis vector
$$
0\neq w_0\;\in\;
\bigwedge^d(\lieh/\liel)_{\QQ_K(\mu)}.
$$
The choice of $w_0$ amounts to choosing an isomorphism of vector spaces
$$
\bigwedge^d(\lieh/\liel)_{\QQ_K(\mu)}^\vee\;\to\;\QQ_K(\mu),\quad
\omega\;\mapsto\;\omega(w_0).
$$
The left hand side is a one-dimensional $L$-module via the adjoint action on
$\lieh/\liel$, and we furnish the right hand side with an action of $L$ such that the above map becomes $L$-linear. The resulting one-dimensional $L$-module is denoted $\mathcal L$. It is of finite order: $\mathcal L^{\otimes 2}\;\cong\;{\bf 1}$.

Assume that the character $\mathcal N^{\otimes k}$ is critical for $M(\mu)^\vee$ and fix a non-zero element
$$
0\neq\xi_k\;\in\;\Hom_H(M(\mu)^\vee,\mathcal N^{\otimes k})_{\QQ_K(\mu)}.
$$
By \cite[Theorem 2.3]{kastenschmidt2008} (or remark \ref{rmk:twistinvarianceatinfinity}), all finite order twists
$$
\chi\;=\;\mathcal N^{\otimes k}\otimes\sgn_{\infty}^\delta,\quad\delta\in\{0,1\}^{r_F^\RR},
$$
are critical quasi-characters of $H(\RR).$ Now for each such $\chi$ and any
$$
\lambda\;\in\;\Hom_{\lieh,C}(A_{\lieq}(\mu),\chi^{-1})_{\QQ_K(\mu)},
$$
the $\QQ_K(\mu)$-rational functionals $\lambda$ and $\xi_k$ induce a $\QQ_K(\mu)$-rational $\pi_0(L)$-equivariant map
$$
\begin{CD}
H(\lambda\otimes\xi_k):\quad
H^d(\lieg,GK^0;A_{\lieq}(\mu)\otimes M(\mu)^\vee)_{\QQ_K(\mu)}@>{\lambda\otimes\xi_k}>>
H^d(\lieh,ML^0;\chi^{-1}[k])_{\QQ_K(\mu)}.
\end{CD}
$$
By Poincar\'e duality, our choice of vector $w_0$ induces an isomorphism
$$
H^d(\lieh,ML^0;\chi^{-1}[k])_{\QQ_K(\mu)}\;\to\;
(\mathcal L\otimes\sgn_{\infty}^\delta)_{\QQ_K(\mu)},
$$
of $\pi_0(L)$-modules. The composition of the latter with $H(\lambda\otimes\xi_k)$, provides us with a $\pi_0(L)$-equivariant map
$$
I(\lambda\otimes\xi_k):\quad
H^d(\lieg,GK^0;A_{\lieq}(\mu)\otimes M(\mu)^\vee)_{\QQ_K(\mu)}\;\to\;
(\mathcal L\otimes\sgn_{\infty}^\delta)_{\QQ_K(\mu)},
$$
which on the level of complexes is given explicitly by
$$
h=\sum_{p=1}^s\omega_p\otimes a_p\otimes m_p\;\mapsto\;\sum_{p=1}^s\omega_p(w_0)\otimes \lambda(a_p)\otimes \xi_k(m_p).
$$

\subsection{The global period relation}

Let $\chi=\otimes_v\chi_v$ be an algebraic Hecke character of $F$ with
$$
\chi_\infty\;=\;\mathcal N^{\otimes k(\chi)}\otimes\sgn_\infty^{\delta(\chi)}
$$
critical. The period $\Omega(\chi_\infty)\in\CC^\times$ in the global special value formula (cf.\ Theorem 1.1 \cite{raghuram2015}),
$$
\frac{\Lambda(\frac{1}{2},\Pi\otimes\chi)}
{G(\chi)^{m}\cdot
\Omega(\chi_\infty)}
\;\in\;
\QQ_K(\mu,\chi)
$$
arises as follows. We fix for each signature $\delta\in\{0,1\}^{r_F^\RR}$ a generator
$$
h_{\delta}
\;\in\;H^d(\lieg,GK^0;A_{\lieq}(\mu)\otimes M(\mu)^\vee)_{\QQ_K(\mu)}\;=\;\QQ_K(\mu)[\pi_0(L)]
$$
of the generalized $\sgn_\infty^{\delta}$-eigenspace for the action of $\pi_0(L)=\pi_0(K_n)$. In our application to $\chi$ we'll specialize later to the $\delta$ satisfying the compatibility condition
\begin{equation}
\sgn_\infty^{\delta}\;=\;\mathcal L\otimes \sgn_\infty^{\delta(\chi)+k(\chi)}.
\label{eq:chidelta}
\end{equation}
Combined with a choice of factorizable Whittaker vector
$$
t^{(\infty)}\;=\;\otimes_v t_v\;\in\;
\mathscr W(\Pi^{(\infty)},\psi^{(\infty)}),
$$
at the finite places $v\nmid\infty$, the cohomological vector $h_{\delta}$ gives rise to a cohomology class
$$
t_{\delta}\;:=\;
\sum_{p=1}^s \omega_p\otimes (\iota_\infty(a_p)\otimes t^{(\infty)})\otimes m_p\;\in\;
H^d(\lieg,GK^0;\mathscr W(\Pi,\psi)^{(K)}\otimes M(\mu)^\vee)_{\CC}.
$$
Inverse Fourier transform turns this class into an automorphic cohomology class
$$
\vartheta_{\delta}\;:=\;
\sum_{p=1}^s \omega_p\otimes \Theta(\iota_\infty(a_p)\otimes t^{(\infty)})\otimes m_p\;\in\;
H^d(\lieg,GK^0;\Pi^{(K)}\otimes M(\mu)^\vee)_{\CC}.
$$
Via the realization
$$
\Pi\;\subseteq\;L_0^2(G(\QQ)Z(\Adeles)\backslash G(\Adeles);\omega_\Pi),
$$
this class gives then rise to a global cohomology class
$$
c(t_{\delta})\;\in\;H_{\rm c}^d(G(\QQ)\backslash G(\Adeles)/GK(\RR)^0K^{(\infty)};\underline{M(\mu)^\vee}_\CC),
$$
with coefficients in the local system associated to $M(\mu)^\vee$, and a suitable compact open $K^{(\infty)}$ which is small enough that $t'$ is $K^{(\infty)}$-invariant and additonally the underlying orbifold is a manifold.

Now the cohomology with compact support carries a natural $\QQ_K(\mu)$-structure, inherited from the $\QQ_K(\mu)$-structure of the rational representation $M(\mu)$. By the work of Clozel \cite{clozel1990}, we know that $\Pi^{(\infty)}$ is defined over its field of rationality $\QQ(\Pi)$. Since the map $t\mapsto c(t)$ is Hecke-equivariant, and since $\Pi^{(\infty)}$ occurs in degree $d$ with multiplicity one by Matsushima's Formula, we may, under the assumption that $t_{\delta}$ is chosen in the natural $\QQ(\Pi)$-rational structure of $\mathscr W(\Pi^{(\infty)},\Psi^{(\infty)})$ (cf.\ \cite{harder1983,mahnkopf2005,raghuramshahidi2008}), renormalize $c(t')$ via a scalar $\Omega(t_{\delta})\in\CC^\times$, such that
\begin{equation}
\Omega(t_{\delta})\cdot c(t_{\delta})\;\in\;H_{\rm c}^d(G(\QQ)\backslash G(\Adeles)/GK(\RR)^0K^{(\infty)};\underline{M(\mu)}_{\QQ(\Pi)}).
\label{eq:normalizedrationalclass}
\end{equation}

To each algebraic Hecke character $\chi$ over $F$, that we interpret as a character of $H(\Adeles)$ via composition with the determinant, we may associate a cohomology class as follows. We denote the corresponding character of $H$ as $(k(\chi))$. Then attached to $\chi$ is a rational cohomology class
$$
c_\chi\;\in\;
H^0(H(\QQ)\backslash H(\Adeles)/L(\RR)^0L^{(\infty)}(\chi);\underline{(k(\chi))}_{\QQ(\chi)}),
$$
where $L^{(\infty)}(\chi)\subseteq K^{(\infty)}\cap H(\Adeles^{(\infty)})$ is a compact open such that the finite order character $\chi$ factors over $\det(L^{(\infty)}(\chi))$.

The natural map
$$
H_{\rm c}^d(G(\QQ)\backslash G(\Adeles)/GK(\RR)^0K^{(\infty)};\underline{M(\mu)^\vee}_{\QQ(\Pi)})\;\to\;
$$
$$
H_{\rm c}^d(H(\QQ)\backslash H(\Adeles)/L(\RR)^0L^{(\infty)};\underline{M(\mu)^\vee}_{\QQ(\Pi)})\;\to\;
$$
$$
\begin{CD}
@>(-)\cup c_\chi>>
H_{\rm c}^d(H(\QQ)\backslash H(\Adeles)/L(\RR)^0L^{(\infty)};\underline{M(\mu)^\vee\otimes(k(\chi))}_{\QQ(\Pi,\chi)}),
\end{CD}
$$
together with Poincar\'e duality for the right hand side, induces the modular symbol
$$
H_{\rm c}^d(G(\QQ)\backslash G(\Adeles)/GK(\RR)^0K^{(\infty)};\underline{M(\mu)^\vee}_{\QQ(\Pi)})\;\to\;
H^0(\Gamma; M(\mu)^\vee\otimes(k(\chi))_{\QQ(\Pi)}),
$$
where $\Gamma\subseteq H(\QQ)$ is the arithmetic subgroup corresponding to $L^{(\infty)}$.

Composition of the modular symbol with $\xi_k$ provides us with a $\pi_0(L)$-equivariant map
$$
H_{\rm c}^d(G(\QQ)\backslash G(\Adeles)/GK(\RR)^0K^{(\infty)};\underline{M(\mu)}_{\QQ(\Pi)})\;\to\;
(\mathcal N^{\otimes k+k(\chi)})_{\QQ(\Pi)}.
$$
The image of the normalized class \eqref{eq:normalizedrationalclass} under this map is essentially the algebraic part of the $L$-value $L(\frac{1}{2}+k,\Pi\otimes\chi)$. To be more precise, the image of $\Omega(t_\delta)^{-1}\cdot c(t_{\delta})$ under this map computes the global integral
$$
\Omega(t_\delta)^{-1}\cdot
\sum_{p=1}^r
\omega_p(w_0)\cdot I(\frac{1}{2}+k,\chi,\Theta(\iota_\infty(a_p)\otimes t^{(\infty)}))\cdot\xi_k(m_p)\;\in\;\mathcal N^{\otimes k}_\CC,
$$
an expression that vanishes whenever the compatibility condition \eqref{eq:chidelta} is violated, i.e.\ we may assume that
$$
\sgn_\infty^{\delta+k}\;=\;\mathcal L\otimes \sgn_\infty^{\delta(\chi)+k(\chi)}.
$$
By the non-archimedean period relation calculated by Raghuram-Shahidi in \cite[Theorem 4.1]{raghuramshahidi2008}, we know that we may find a $\QQ(\pi,\chi)$-rational $t^{(\infty)}\in\mathscr W(\Pi^{(\infty)},\psi^{(\infty)})$, such that
$$
e^{(\infty)}(\chi^{(\infty)}|\cdot|_{\Adeles^{(\infty)}}^s,t^{(\infty)})\;=\;G(\chi)^{m}
L(s,\Pi^{(\infty)}\otimes\chi^{(\infty)}),
$$
where $m$ is as before. In other words, the image of the corresponding cohomology class $c(t_\delta)$ under the modular symbol and $\xi_k$ computes the value
$$
G(\chi)^{m}\cdot
\Lambda(\frac{1}{2}+k,\Pi\otimes\chi)\cdot
\sum_{p=1}^r
\omega_p(w_0)\cdot e_\infty(\frac{1}{2}+k,\chi_\infty,\iota_\infty(a_p))\cdot\xi_k(m_p).
$$
Now the values $a_p$ lie in the same $\QQ_K(\mu)$-rational subspace as the good test vector $t$, and by Corollary \ref{cor:criticalrelation} we conclude the proof of the desired period relation. We obtain the global

\begin{theorem}\label{thm:globalrankinrelation}
  Let $n\geq1.$ If $n\geq 3$ or $n\geq 2$ if $F$ admits a complex place, assume the validity of Conjecture \ref{conj:rankinfunctional0} or of Conjecture \ref{conj:rankinfunctional} for a balanced weight $\mu$ of $G.$
  
  Let $(\Pi_1,\Pi_2)$ be a pair of regular algebraic irreducible cuspidal automorphic representations of $\GL_{n+1}(\Adeles_F)$ and $\GL_n(\Adeles_F)$ respectively. Assume that $\Pi_\infty=\Pi_{1}\widehat{\otimes}\Pi_{2}$ is of balanced cohomological weight $M.$ Denote by $s_0=\frac{1}{2}+k$ the left most critical value of the Rankin-Selberg $L$-function $L(s,\Pi_1\times\Pi_2)$ (such an $s_0$ exists).

  Then there exist non-zero periods $\Omega_\pm$, numbered by the characters $\pm$ of $\pi_0((F\otimes_\QQ\RR)^\times)$, such that for each critical half integer $s_1=\frac{1}{2}+k_1$ for $L(s,\pi\times\sigma)$, and each finite order Hecke character
$$
\chi:\quad F^\times\backslash\Adeles_F^\times\to\CC^\times,
$$
we have
$$
\frac{L(s_1,\Pi_1\times(\Pi_2\otimes\chi)}
{G(\overline{\chi})^{\frac{(n+1)n}{2}}\Omega_{(-1)^{k_1}\sgn\chi}}
\;\in\;
\frac{L(s_0,\Pi_{1,\infty}\times\Pi_{2,\infty})}
{L(s_1,\Pi_{1,\infty}\times\Pi_{2,\infty})}\cdot
i^{k_1[F:\QQ]\frac{(n+1)n}{2}}
\QQ(\Pi_1,\Pi_2,\chi).
$$
Furthermore, for every $\tau\in\Aut(\CC/\QQ)$,
$$
i^{-k_1[F:\QQ]\frac{(n+1)n}{2}}
\frac{L(s_1,\Pi_{1,\infty}\times\Pi_{2,\infty})}
{L(s_0,\Pi_{1,\infty}\times\Pi_{2,\infty})}\cdot
\frac{L(s_1,\Pi_1^\tau\times(\Pi_2^\tau\otimes\chi^\tau)}
{G(\overline{\chi^\tau})^{\frac{(n+1)n}{2}}\Omega_{(-1)^{k_1}\sgn\chi^\tau}}
\;=\quad\quad
$$
$$
\quad\quad
\left(
i^{-k_1[F:\QQ]\frac{(n+1)n}{2}}
\frac{L(s_1,\Pi_{1,\infty}\times\Pi_{2,\infty})}
{L(s_1,\Pi_{1,\infty}\times\Pi_{2,\infty})}\cdot
\frac{L(s_0,\Pi_1\times(\Pi_2\otimes\chi)}
{G(\overline{\chi})^{\frac{(n+1)n}{2}}\Omega_{(-1)^{k_1}\sgn\chi}}
\right)^\tau.
$$
\end{theorem}

\section{Deligne's Conjecture}\label{sec:deligne}

\subsection{Motives and Deligne's Conjecture for $M(\Pi_1)\otimes M(\Pi_2)\otimes\chi$}

Attached to $\Pi_1$ and $\Pi_2$ are conjectural simple motives $M(\Pi_1)$ and $M(\Pi_2)$ of ranks $n+1$ and $n$ over $F,$ characterized by suitable identities of $L$-functions, cf.\ \cite{clozel1990}. By Grothendieck's Standard Conjectures, we expect the category of motives to be tannakian, i.e.\ in particular it is expected to be a rigid tensor category, which allows us to consider the tensor product $M(\Pi):=M(\Pi_1)\otimes M(\Pi_2).$ Assuming this, we expect to have an identity
\begin{equation}
L(s-\frac{2n-1}{2},\Pi_1\times\Pi_2)\;=\;L(s,M(\Pi_1)\otimes M(\Pi_2)).
\label{eq:lidentity}
\end{equation}

Attached to $M(\Pi_1)\otimes M(\Pi_2)$ is a collection of realizations. Apart from the $\ell$-adic realizations for which we have candidates by recent work of Harris-Lan-Taylor-Thorne and Scholze whenever $F$ is totally real or a CM field, the Betti (or Hodge) and the de Rham realizations feature prominently in Deligne's Conjecture on special values \cite{deligne1979}.

Using these, Deligne defines for each finite order character $\varepsilon$ of
$$
\pi_0((F\otimes_\QQ\RR)^\times)\;=\;\pi_0(L),
$$
complex periods
\begin{equation}
c_\varepsilon\;=\;\prod_v c_v^{\epsilon({\bf1}_v)}(M(\Pi_1)\otimes M(\Pi_2)),
\label{eq:periodproduct}
\end{equation}
well defined modulo the field of rationality $E=\QQ(\Pi_1,\Pi_2),$ which by our identification of $L$-functions \eqref{eq:lidentity} comes with a fixed embedding $E\to\CC.$ Furthermore, Deligne defines natural numbers
$$
d^\varepsilon\;=\;\sum_v d_v^{\epsilon({\bf1}_v)}
$$
as dimensions of certain cohomology spaces. In the case at hand, $d^\varepsilon$ is expected to be independent of $\varepsilon$ and given by
$$
d^\varepsilon\;=\;(r_F^\RR+2r_F^\CC)\cdot\frac{(n+1)n}{2}\;=\;[F:\QQ]\cdot\frac{(n+1)n}{2}.
$$
In this setting, Deligne's Conjecture \cite[Conjecture 2.8]{deligne1979} reads in our case (cf.\ (5.1.8) and section 8 in loc.\ cit., see also \cite{blasius1997} for the behaviour under finite order character twists, as well as \cite{blasius1987} for the periods attached to tensor products)
\begin{conjecture}[Deligne]\label{conj:deligne}
For each $k\in\ZZ$ critical for $M(\Pi_1)\otimes M(\Pi_2)$ and each finite order character $\chi:\Gal(\overline{F}/F)\to\CC^\times$ we have
$$
\frac{L(k,M(\Pi_1)\otimes M(\Pi_2)\otimes\chi)}
{G(\overline{\chi})^{[F:\QQ]\cdot\frac{(n+1)n}{2}}(2\pi i)^{k[F:\QQ]\cdot\frac{(n+1)n}{2}}c_{(-1)^k\sgn\chi}}\;\in\;E(\chi).
$$
\end{conjecture}

\subsection{Compatibility with Deligne's Conjecture}

A direct computation (cf.\ Proposition 1.5 in \cite{kastenschmidt2008} and also the formula given on p.\ 219 and p.\ 220 of loc.\ cit.), it is easy to see that
\begin{equation}
\frac{L(s_1,\pi_\infty\times\sigma_\infty)}
{L(s_0,\pi_\infty\times\sigma_\infty)}\;\in\;(2\pi)^{(s_0-s_1)[F:\QQ]\frac{(n+1)n}{2}}\QQ^\times.
\label{eq:picalc}
\end{equation}

Incorporating the results from the previous section into Theorem \ref{thm:globalrankinrelation} we obtain

\begin{theorem}\label{thm:globalrankinperiods}
  Let $F$ be a number field and let $n\geq 1.$ If $n\geq 3$ or $n\geq 2$ if $F$ admits a complex place, assume the validity of Conjecture \ref{conj:rankinfunctional0} or of Conjecture \ref{conj:rankinfunctional} for a balanced weight $\mu$ of $G.$

  Let $\Pi_1,$ $\Pi_2$ be a pair regular algebraic irreducible cuspidal automorphic representation of $\GL_{n+1}(\Adeles_F)$ and $\GL_n(\Adeles_F)$ respectively such that $\Pi_{1}\widehat{\otimes}\Pi_{2}$ is of balanced cohomological weight $\mu.$ Then there exist non-zero periods $\Omega_\pm$, numbered by the $2^{r_F^\CC}$ characters $\pm$ of $\pi_0((F\otimes_\QQ\RR)^\times)$, such that for each critical half integer $s_0=\frac{1}{2}+k$ for $L(s,\Pi_1\times\Pi_2),$ and each finite order Hecke character
$$
\chi:\quad F^\times\backslash\Adeles_F^\times\to\CC^\times,
$$
we have, in accordance with Deligne's Conjecture \ref{conj:deligne},
$$
\frac{L(s_0,\Pi_1\times(\Pi_2\otimes\chi))}
{G(\overline{\chi})^{\frac{(n+1)n}{2}}(2\pi i)^{kr_F\frac{(n+1)n}{2}}
\Omega_{(-1)^{k}\sgn\chi}}
\;\in\;
\QQ_K(\Pi_1,\Pi_2,\chi).
$$
Furthermore, this expression is $\Aut(\CC/\QQ_K)$-equivariant.
\end{theorem}

Recall that $\QQ_K=\QQ$ whenever $F$ is totally real or a CM field.

\begin{corollary}\label{cor:deligneverification}
  Assume that $F$ is totally real or a CM field. Then under the assumptions of Theorem \ref{thm:globalrankinperiods} Deligne's Conjecture \ref{conj:deligne} for the conjectural motive $M(\Pi_1)\otimes M(\Pi_2)\otimes\chi$, for a rank $1$ Artin motive $\chi,$ is equivalent to the statement
$$
\frac{
\Omega_{(-1)^{j}\sgn\chi}}
{(2\pi i)^{r_F\frac{(n+1)n^2}{2}}c_{(-1)^{j+n}\sgn\chi}}
\;\in\;\QQ(\Pi_1,\Pi_2)^\times,\quad j\in\{0,1\}.
$$
\end{corollary}

The statement of the corollary extends to arbitrary $F$ if we adjoint $\QQ_K$ to the coefficient field $E.$

\begin{remark}
  The statement of Corollary \ref{cor:deligneverification} leaves out valuable finer structure, i.e.\ the conjectural implications that result from the fact that $M(\Pi_1)\otimes M(\Pi_2)$ is a tensor product (cf.\ \cite{blasius1987}), and also the finer description of Deligne's periods $c_\pm$ in terms of products of the periods $c_v^\pm(M(\Pi_1)\otimes M(\Pi_2))$ attached to (real) archimedean embeddings as in \eqref{eq:periodproduct}  (cf.\ \cite{blasius1997}). These finer results are known for Hilbert modular forms \cite{harris1993}.
  \end{remark}

\bibliographystyle{plain}

\end{document}